\documentclass[11pt,a4paper]{amsart}

\usepackage{amsmath, amssymb} \newfont {\cyr} {wncyr10} 
\usepackage{amsfonts} \usepackage{amsmath} \usepackage{amssymb} \usepackage{setspace}
\usepackage{multicol}
\usepackage{color}
\usepackage{caption}

\usepackage{mathtools}
\usepackage{fancyref}

\newcommand*{\fancyrefthmlabelprefix}{thm}
\frefformat{plain}{\fancyrefthmlabelprefix}{Theorem \textup{#1}}

\newcommand*{\fancyrefequationlabelprefix}{eq}
\frefformat{plain}{\fancyrefequationlabelprefix}{\textup{(#1)}}

\newcommand*{\fancyrefsectionlabelprefix}{sec}
\frefformat{plain}{\fancyrefsectionlabelprefix}{Section~\textup{#1}}

\newcommand*{\fancyrefnotlabelprefix}{not}
\frefformat{plain}{\fancyrefnotlabelprefix}{Notation~\textup{#1}}

\newcommand*{\fancyrefhyplabelprefix}{hyp}
\frefformat{plain}{\fancyrefhyplabelprefix}{Hypothesis~\textup{#1}}

\newcommand*{\fancyreflemmalabelprefix}{lem}
\frefformat{plain}{\fancyreflemmalabelprefix}{Lemma~\textup{#1}}

\newcommand*{\fancyrefdeflabelprefix}{def}
\frefformat{plain}{\fancyrefdeflabelprefix}{Definition~\textup{#1}}

\newcommand*{\fancyrefproplabelprefix}{prop}
\frefformat{plain}{\fancyrefproplabelprefix}{Proposition~\textup{#1}}

\newcommand*{\fancyrefcorlabelprefix}{cor}
\frefformat{plain}{\fancyrefcorlabelprefix}{Corollary~\textup{#1}}

\renewcommand*{\fancyreftablabelprefix}{tab}
\frefformat{plain}{\fancyreftablabelprefix}{Table~\textup{#1}}

\renewcommand*{\fancyrefcorlabelprefix}{cor}
\frefformat{plain}{\fancyrefcorlabelprefix}{Corollary~\textup{#1}}

\newcommand*{\fancyrefchplabelprefix}{chp}
\frefformat{plain}{\fancyrefchplabelprefix}{Chapter~\textup{#1}}

\newcommand*{\fancyrefclaimlabelprefix}{claim}
\frefformat{plain}{\fancyrefclaimlabelprefix}{Claim~\textup{#1}}

\newcommand*{\fancyrefclmlabelprefix}{clm}
\frefformat{plain}{\fancyrefclmlabelprefix}{\textup{#1}}

 \newtheorem{Def}{Definition}[section]

\newtheorem{lemma}[Def]{Lemma}

\newtheorem{proposition}[Def]{Proposition} 
 \newtheorem{Equ}{}

\newcounter{claim}[Def]

\newtheorem*{theorem*}{Main Theorem}
\newtheorem*{corollary*}{Corollary}
\newtheorem*{notation*}{Notation}

 \def \Aut {\mbox {\rm Aut}} \def \Sym {\mbox {\rm Sym}} \def \Mat {\mbox {\rm Mat}}  \def \Alt {\mbox {\rm Alt}}   \def \Syl {\mbox {\rm Syl}}  \def \Q {\mbox {\rm Q}}\def \Out {\mbox {\rm Out}}\def \ov {\overline}
\def \syl{\mathrm {Syl}}
\def \GammaSL{\Gamma\mathrm{SL}} \def \qedc {$\hfill \blacksquare$\newline}
\def \wt{\widetilde}\def \OO {\mathrm O}

\mathchardef\tnode="020E 

\def\arc{
  \hbox{\kern -0.15em
  \vbox{\hrule width 2.5em height 0.6ex depth -0.5 ex}
  \kern -0.33em}}

\def\darc{
  \rlap{\lower0.2ex\arc}{\raise0.2ex\arc}}

\def\tarc{
 \rlap{\rlap{\lower0.4ex\arc}{\raise0.4ex\arc}}{\arc}}

\def\stroke#1{
  \kern 0.05em
  \rlap\arc{{\textstyle{#1}}\atop\phantom\arc}
  \kern -0.22em}

\def\dstroke#1{
  \kern 0.05em
  \rlap\darc{{\textstyle{#1}}\atop\phantom\darc}
  \kern -0.22em}

\def\centerscript#1{
  \setbox0=\hbox{$\tnode$}
  \hbox to \wd0{\hss$\scriptstyle{#1}$\hss}}

\def\node{
  \def\super{}
  \def\sub{}
  \futurelet\next\dolabellednode}

  \let\sp=^
  \let\sb=_

  \def\dolabellednode{%
    \ifx\next\sb\let\next\getsub
    \else
      \ifx\next\sp\let\next\getsuper
      \else\let\next\donode
      \fi
    \fi
    \next}

  \def\getsub_#1{\def\sub{#1}\futurelet\next\dolabellednode}

\def\getsuper^#1{\def\super{#1}\futurelet\next \dolabellednode}

  \def\donode{%
\rlap{$\mathop{\phantom\tnode}\limits_{\centerscript{\sub}}
    ^{\centerscript{\super}}$}\tnode}

\def\varcdn{
  \kern -0.03em\vbox{\kern -0.5ex
  \hbox to \wd0{\hss\vrule width 0.04em depth 5.8ex\hss}
  \kern -0.3ex
  \hbox{$\tnode$}}}

\def\m24{\node\stroke{c}\node\darc\node\arc\node} \def\tc5{\node\arc\node\arc\node\arc\node\dstroke{\sim}\node}

\def\Co1{\node\arc\node\arc\node\dstroke{\sim}\node} 

\newcommand{\varcdnl}[1]{ 
  \kern -0.03em\vbox{\kern -0.5ex
  \hbox to \wd0{\hss\vrule width 0.04em depth 5.8ex\hss}
  \kern -0.3ex
  \hbox{$\tnode^{#1}$}}}

 \def\m24l{\node^1\stroke{c}\node^2\darc\node^3\arc\node^4}
\def\tc5l{\node^1\arc\node^2\arc\node^3\arc\node^4\dstroke{\sim}\node^5}
\def\U43l{\node^1\darc\node^2\darc\node^3} \def\F2l{\node^1\arc\node^2\arc\node^3\stroke{M_{22}}\node^4}
 
 \def\Co2l{\node^1\arc\node^2\stroke{M_{22}}\node^3}
\def\J4l{\node^1\arc\node^2\stroke{{\widetilde{M_{22}}}}\node^3}
 \def\sm24{\node^2\arc\node^1\dstroke{\sim}\node^0}
\def\sCo1{\node^4\arc\node^2\arc\node^1\dstroke{\sim}\node^0}

\def\nodef{
  \def\super{}
  \def\sub{}
  \futurelet\next\dolabellednodef}

  \let\sp=^
  \let\sb=_

  \def\dolabellednodef{%
    \ifx\next\sb\let\next\getsubf
    \else
      \ifx\next\sp\let\next\getsuperf
      \else\let\next\donodef
      \fi
    \fi
    \next}

\def\getsubf_#1{\def\sub{#1}\futurelet\next\dolabellednodef}

\def\getsuperf^#1{\def\super{#1}\futurelet\next \dolabellednodef}

  \def\donodef{%
\rlap{$\mathop{\phantom\tnodef}\limits_{\centerscript{\sub}}
    ^{\centerscript{\super}}$}\tnodef}

\def\varcdnf{
  \kern -0.03em\vbox{\kern -0.5ex
  \hbox to \wd0{\hss\vrule width 0.04em depth 5.8ex\hss}
  \kern -0.3ex
  \hbox{$\tnodef$}}}

\newcommand{\e}{\end{document}} 

\def\multdef#1#2{\actdef{#1}{#2}}

\def\actdef#1#2#3{%
   \if #3. \def\next{}%
   \else\expandafter\def\csname #1#3\endcsname{{#2{#3}}}%
        \def\next{\multdef{#1}{#2}}
   \fi%
 \next}

\multdef{}{\operatorname}
          {lcm}{Hom}{B}{Aut}{Fix}{Inn}{GF}{Ly}{A}
          {McL}{Sym}{Mat}{Cut}{Jump}{Alt}{Stab}
          {Syl}{D}{SD}{rad}{sgn}{id}{End}{tr}
          {Frob}{Ann}{rank}{Span}{J}{Int}
          {Irr}{IBr}{Bl}{Br}{Dih}{Supp}{FSym}
          {lead}{Diag}{GL}{SL}{ad}{orb}{T}{M}{nat}{PSL}
          {PSU}{PSp}{Sp}{G}{SU}{Sz}{F}{GU}{J}{U}{Co}{Suz}{E}{F}{L}{HN}{He}
{PGL}.

\begin{document}

\title {The Local Structure Theorem: The wreath product case}

 \author{Chris Parker}
  \author{Gernot Stroth}

\address{Chris Parker\\
School of Mathematics\\
University of Birmingham\\
Edgbaston\\
Birmingham B15 2TT\\
United Kingdom} \email{c.w.parker@bham.ac.uk}

\address{Gernot Stroth\\
Institut f\"ur Mathematik\\ Universit\"at Halle - Wittenberg\\
Theordor Lieser Str. 5\\ 06099 Halle\\ Germany}
\email{gernot.stroth@mathematik.uni-halle.de}

\maketitle

\begin{center} Dedicated to the memory of Kay Magaard  \end{center}

\begin{abstract} Groups with a large $p$-subgroup, $p$ a prime,  include almost all of the groups of Lie type in characteristic $p$ and so the study of such groups adds to our understanding of the finite simple groups.  In this article we study a special class of  such groups which appear as wreath product cases of the Local Structure Theorem \cite{stru}.
\end{abstract}

\section{Introduction}
Throughout this article $p$ is a   prime and $G$ is a finite group.
We say that $L\le G$ has \emph{characteristic $p$} if $$C_G(O_p(L))\leq O_p(L).$$For $T$ a non-trivial $p$-subgroup of $G$,
the subgroup $N_G(T)$ is called a \emph{ $p$-local} subgroup of $G$. By definition $G$ has \emph{local characteristic $p$} if all $p$-local subgroups of $G$ have characteristic $p$ and $G$ has \emph{parabolic characteristic $p$} if all $p$-local subgroups containing a Sylow $p$-subgroup of $G$ have characteristic $p$.

A group $K$ is called a \emph{${\mathcal{K}}$-group} if all its composition factors are from the known finite simple groups. So, if $K$ is a simple ${\mathcal{K}}$-group, then $K$ is a cyclic group of prime order, an alternating group, a simple group of Lie type or one of the 26 sporadic simple groups.  A group $G$ is a
\emph{${\mathcal{K}}_p$-group}, provided  all subgroups of all $p$-local subgroups of $G$ are ${\mathcal{K}}$-groups.
This paper is
part of a programme to investigate the structure of certain $\mathcal K_p$-groups. See \cite{ov,stru} for  an overview of the project.

Of fundamental importance to the development of the programme  are large subgroups of $G$:  a
$p$-subgroup $Q$  of $G$ is  \emph{large} if
 \begin{itemize}
 \item[(i)] $C_G(Q) \leq Q$; and
 \item[(ii)]  $N_G(U)
\leq N_G(Q)$ for all  $1 \neq U \leq C_G(Q)$. \end{itemize}
\noindent For example, if $G$ is a simple group of Lie type defined in characteristic $p$,  $S\in \Syl_p(G)$ and $Q=O_p(C_G(Z(S)))$,
then $Q$ is a large subgroup of $G$ unless there is some degeneracy in the Chevalley commutator relations which define $G$. This means that $Q$ is a large subgroup of $G$ unless  $G$ is  one of $\Sp_{2n}(2^k)$, $n\geq 2$, $\F_4(2^k)$ or $\G_2(3^k)$. \medskip

If $Q$ is a large subgroup of $G$, then it is easy to see that $O_p(N_G(Q))$ is also a large $p$-subgroup of $G$. Thus we  also assume that
\begin{itemize}
\item[(iii)] $Q = O_p(N_G(Q))$.
\end{itemize}
One of the consequences of $G$ having a large $p$-subgroup is that $G$ has  parabolic characteristic $p$. In fact any $p$-local subgroup of $G$ containing $Q$ is of characteristic $p$ \cite[Lemma 1.5.5 (e)]{stru}.
Further, if  $Q \le S \in \syl_p(G)$, then $Q$ is weakly closed in $S$ with respect to $G$ ($Q$ is the unique $G$-conjugate of $Q$ in $S$)  \cite[Lemma 1.5.2 (e)]{stru}.
A significant part of the programme described in \cite{ov} aims to determine the groups which possess a large $p$-subgroup. This endeavour extends and generalizes earlier work of Timmesfeld and others in the original proof of the classification theorem where groups with a so-called large extraspecial $2$-subgroup were investigated. The state of play at the moment is that the Local Structure Theorem has been completed and published \cite{stru}. To describe this result we need some further notation.

 For
a finite group $L$,   $Y_L$ denotes  the unique  maximal elementary abelian normal $p$-subgroup of $L$ with
$O_p(L/C_L(Y_L)) = 1$. Such a subgroup exists  \cite[Lemma 2.0.1(a)]{ov}. From now on assume that $G$ is a finite $\mathcal K_p$-group, $S$ a
Sylow $p$-subgroup  of $G$ and  $Q$ a large $p$-subgroup of $G$ with $Q\leq S$ and  $Q=O_p(N_G(Q))$.
We define $$\mathcal L_G(S) = \{ L \le G\mid S \le L, O_p(L) \not= 1, C_G(O_p(L)) \leq O_p(L)\}.$$
Under the assumption that $S$ is contained in at least two maximal $p$-local subgroups, for $L \in \mathcal L_G(S)$ with $L \not\le N_G(Q)$, the Local Structure Theorem provides information about  $L/C_L(Y_L)$ and its action on $Y_L$. Given the Local Structure Theorem there are two cases to treat in order to fully understand groups with a large $p$-subgroup. Either there exists $L  \in \mathcal L_G(S)$ with $Y_L \not\leq Q$ or, for all $L \in \mathcal L_G(S)$, $Y_L \leq Q$. Research  in the first case has just started and, for this situation,  this paper addresses  the wreath product scenario  in the Local Structure Theorem \cite[Theorem A (3)]{stru}. This case is separated from the rest because of the special structure of   $L$ and $Y_L$. This structure allows us to use arguments measuring the size of certain subgroups to reduce to three exceptional configurations and  has a distinct flavour from the remaining cases. For instance, the groups which are examples in the wreath product case typically have $Q$ of class 3 whereas in the more typical cases it has class at most 2. The  configurations in the Local Structure Theorem  which are not in the wreath product case and have $Y_L \not \le Q$ will be examined in a separate publication as there are methods which apply uniformly to cover many possibilities at once.
Contributions to the  $Y_L \leq Q$ for all $L \in \mathcal{L}_G(S)$ are the subject of \cite{PPS}.

  For $L \in \mathcal L_G(S)$ with $Q$ not normal in $L$ we set  $$L^\circ = \langle Q^L \rangle,   \ov   {L} = L/C_L(Y_L)\text{ and } V_L=   [Y_L,L^\circ]$$ and use this notation throughout the paper.
Set $q=p^a$.  We recall from \cite[Remark A.25]{stru} the definition of  a \emph{natural wreath $\SL_2(q)$-module} for the group $X$ with respect to $\mathcal K$:
 suppose that $X$ is a group, $V$ is a faithful $X$-module  and $\mathcal K$ is a non-empty $X$-invariant set of subgroups of $X$. Then $V$ is a \emph{natural $\SL_2(q)$-wreath product module} for $X$ with respect to $\mathcal K$ if and only if
$$V = \bigoplus_{K \in \mathcal K}[V,K]\text{ and } \langle \mathcal K \rangle = \bigtimes_{K \in \mathcal K} K,$$
and, for each $K \in \mathcal K$, $K \cong \SL_2(q)$ and $[V,K]$ is the natural $\SL_2(q)$-module for $K$.

We now  describe the wreath  product case in \cite[Theorem A (3)]{stru}. For $L \in \mathcal L_G(S)$ with $L \not \le N_G(Q)$, $L$ is in the \emph{wreath product case} provided
   \begin{itemize}
   \item there exists a unique $\ov L$-invariant  set $\mathcal K$ of subgroups of $\ov L$ such that $V_L  $ is a natural $\SL_2(q)$-wreath product module for $\ov L$ with respect to $\mathcal K$.
       \item   $\ov  {L^\circ}= O^p( \langle  \mathcal K\rangle)\ov Q$   and $Q$ acts transitively on $\mathcal K$ by conjugation.
\item $Y_L = V_L$ or $p = 2$, $|Y_L : V_L| = 2$, $\ov{L^\circ} \cong \SL_2(4)$ or $\Gamma \SL_2(4)$ and $V_L \not\leq Q$.
   \end{itemize}
    We say that $\ov L$ is \emph{properly wreathed} if $|\mathcal K| > 1$.

There are  overlaps between the wreath product case  and some other divisions  in the Local Structure Theorem.

If   $\ov{L^\circ} \cong \SL_2(q)$ with $V_L = Y_L$,  then this situation can be inserted in the linear case of \cite[Theorem A (1)]{stru} by including  $n=2$ is that case. Suppose  that $|\mathcal {K}| = 2$ and $K \cong \SL_2(2)$. If $\ov{Q}$ is a fours group, then, as $\ov Q$ conjugates $\ov{K_1}$ to $\ov {K_2}$, $$\ov{L^\circ} \cong \Omega^+_4(2) \cong \SL_2(2) \times \SL_2(2)$$ and $Y_L$ is the tensor product module. This is an example in the tensor product case of \cite[Theorem A (6)]{stru}.
 We declare $L$  to be in the  \emph{unambiguous wreath product case} if these two \emph{ambiguous} configurations  do not occur.
The ambiguous cases  will be handled in a more general setting in a forthcoming paper mentioned earlier.

 \begin{theorem*}  Suppose that $p$ is a prime, $G$ is  a finite group, $S$ a Sylow
$p$-subgroup of $G$ and $Q\le S $ is a large $p$-subgroup of $G$ with $Q=O_p(N_G(Q))$.
If there exists $L \in \mathcal L_G(S)$  with $L$ in the unambiguous wreath product case and  $V_L \not \le Q$, then $G \cong \Mat(22)$, $\Aut(\Mat(22))$, $\Sym(8)$, $\Sym(9)$ or $\Alt(10)$.
\end{theorem*}

The proof of this theorem splits into four parts. First, in Section 3, we show that in the properly wreathed case we must have $q=|\mathcal K|=2$ and, as $L$  is unambiguous,  $\ov{S} = \ov{Q}\cong \Dih(8)$ and $\ov{L^\circ} \cong \OO_4^+(2)$. If $|\mathcal K|=1$, we show that $\ov{L^\circ} \cong \Gamma \SL_2(4)$  or $\SL_2(4)$ and $V_L$ is the natural module with $|Y_L:V_L|\le 2$, where, if $\ov{L^\circ} \cong \SL_2(4)$,  $|Y_L : V_L| = 2$ holds. In the following three  sections,  we determine the groups corresponding to these three cases. Finally the Main Theorem  follows by combining Propositions~\ref{prop:unambiguous}, \ref{prop:O4}, \ref{prop:S5-1} and \ref{prop:L24}.
\\\\
In   \cite{PPS} the authors proved that the unambiguous wreath product case does not lead to examples if for all $L \in \mathcal{L}_G(S)$ we have $Y_L \leq Q$, with the additional assumption that $G$ is of local characteristic $p$.  In this paper we do not make the assumption that $G$ is of local characteristic $p$.
\\
\\
In the Local Structure Theorem there is also a possibility that $L \in \mathcal{L}_G(S)$ is of  weak wreath type. Any such group is contained in one, which is of unambiguous wreath type. A corollary of our theorem is

\begin{corollary*} Suppose that $p$ is a prime, $G$ is  a finite group, $S$ a Sylow
$p$-subgroup of $G$ and $Q\le S $ is a large $p$-subgroup of $G$ with $Q=O_p(N_G(Q))$. If  $L \in \mathcal{L}_G(S)$ is of weak wreath product type, then either $G$ is  as in the Main Theorem   or $V_L  \le Q$.
\end{corollary*}

In addition to the notation already introduced, we will use  the following
\begin{notation*}\label{not:not} For $p$ a prime,  $G$ a group with a large $p$-subgroup $Q= O_p(N_G(Q))$ and $L \in \mathcal{L}_G(S)$, we set $Q_L = O_p(L)$ and assume that $V_L \not \le Q$. Define $D= \langle V_L^{N_G(Q)}\rangle (L\cap N_G(Q)) \in \mathcal L_G(S)$. Furthermore,   set
 $$W = \langle (V_L \cap Q)^{D}\rangle,$$  $$U_L = \langle (W \cap Q_L)^L \rangle$$ and $$Z= C_{V_L}(Q).$$
\end{notation*}

Notice that for $L_0=N_L(S \cap C_L(Y_L))$, we have $L= C_L(Y_L) L_0$ and $C_L(Y_L) \le D$. Further  $$Y_{L_0}= Y_L= \Omega_1(Z(O_p(L_0)))$$ by \cite[Lemma 1.2.4 (i)]{stru}. Since $C_L(Y_L)$ normalizes $Q$, $$L^\circ= \langle Q^L\rangle = \langle Q^{C_L(Y_L)L_0} \rangle=  \langle Q^ { L_0}\rangle = L_0^\circ.$$ Therefore, if $L$ is in the unambiguous wreath product case, then so is  $L_0$. Hence we also  assume that $L= L_0$ and so $$Y_L =\Omega_1(Z(Q_L)).$$

\section{Preliminaries}

In this section we present some lemmas which will be used in the forthcoming sections.

\begin{lemma}\label{lem:splitA6} Suppose that $X$ is a group, $E = O_2(X)$ is elementary abelian of order $16$ and $X/E \cong \Alt(6)$ induces the non-trivial irreducible part of the $6$-point permutation module on $E$. Then $X$ splits over $E$.
\end{lemma}

\begin{proof} Choose $R \leq X$ such that $R/E \cong \Sym(4)$ and $Z(R) = 1$. Let $T \in \syl_3(R)$. As $T$ acts fixed-point freely on $O_2(R)$,  $N_R(T) \cong \Sym(3)$ and so there are involutions in $X /E$.  Hence, as $X/E$ has one conjugacy class of involutions, there are involutions in $O_2(R) \setminus E$. Therefore $O_2(R)/ Z(O_2(R))$ is elementary abelian of order $16$.  Now we consider $O_2(R)$. The fixed-point free action of $T$ on  $O_2(R)/Z(O_2(R))$  implies there is partition of this group into five $T$-invariant subgroups of order 4.  As $T$ acts fixed-point freely on $O_2(R)$ the preimages of all these fours groups are abelian.
 As there are involutions in $O_2(R) \setminus E$, there is a $T$-invariant fours group $F^* \le O_2(R)/Z(O_2(R))$ with $F^* \ne E/Z(O_2(R))$ and such that the preimage $F$ of $F^*$ is elementary abelian of order 16.
 Now the action of $X$ on $E$ shows that for any involution $i \in R \setminus E$ all involutions in the coset $Ei$ are conjugate to $i$ by an element of $E$. Hence all involutions in $O_2(R) \setminus E$ are in $F$. This shows that $F$ is invariant under $N_R(T)$.

 Again there is a partition of $F$ into five groups of order four invariant under $T$. Let $t$ be an involution in $N_R(T)$. Then $|C_F(t)| = 4$, where $|C_{E \cap F}(t)| = 2$. Hence there is some fours group $F_1 \leq F$, $F_1 \not = E \cap F$ and $C_{F_1}(t) \not= 1$. This shows that $F_1$ is normalized by $t$. Then $F_1\langle t \rangle\cong \Dih(8)$ is a complement to $E$. Using a result  of  Gasch\"utz \cite[Theorem 9.26]{GoLyS2} , $X$ splits over $E$.
\end{proof}

The next  lemma is  well-known.
\begin{lemma}\label{lem:Sym5-modules} Suppose that $X \cong \Sym(5)$, $F_1$ and $F_2$ are fours groups of $X$ with $F_1 \le \Alt(5)$ and $V$ is a non-trivial irreducible $\GF(2)X$-module. Then
\begin{enumerate}
\item $V$ is either the non-trivial irreducible part of the permutation module, which is the same as the natural $\OO^-_4(2)$-module, or $V$ is the natural $\Gamma \L_2(4)$-module.
\item $F_1$ acts quadratically on $V$   if and only if $V$ is the natural $\Gamma \L_2(4)$-module.
\item $F_2$ acts quadratically on $V$   if and only if $V$ is the natural $\OO_4^-(2)$-module.
\end{enumerate}
\end{lemma}

\begin{lemma}\label{lem: cent W/[W,Q]}
Suppose that $p$ is a prime, $X$ is a group of characteristic $p$ and $U$ is a normal $p$-subgroup of $X$.
Let $R$ be a  normal subgroup of $X$ with $R\le C_{X}(U/[U,O_p(X)])$.  If $[O_p(X),O^p(R)]\le U$, then $R\le O_p(X)$.
\end{lemma}

\begin{proof}  It suffices to prove that $O^p(R)=1$. Suppose that $n \ge 1$ is such that  $[U,O^p(R)]\le [U,O_p(X);n]$. Then $$[O_p(X),O^p(R)] = [O_p(X),O^p(R),O^p(R)] \le [U,O^p(R)] \le [U,O_p(X);n]$$
  and so $$[O_p(X),O^p(R),U]\le [[U,O_p(X);n],O_p(X)]=[U,O_p(X);n+1].$$
 We also have $$[U,O^p(R),O_p(X)] \le [[U,O_p(X);n],O_p(X)]=[U,O_p(X);n+1]$$ and thus  the Three Subgroups Lemma implies   $$[U,O_p(X),O^p(R)]\le [U,O_p(X);n+1].$$
This yields $$[U,O^p(R)]=[U,O^p(R),O^p(R)]\le [U,O_p(X), O^p(R)]\le [U,O_p(X);n+1].$$  Since $O_p(X)$ is nilpotent, we deduce $[U,O^p(R)]=1$. Hence  $$[O_p(X),O^p(R)] =[O_p(X),O^p(R),O^p(R)]\le  [U,O^p(R)]=1.$$ As  $X$ has characteristic $p$, $O^p(R)=1$ and so $R\le O_p(X)$ as claimed.
\end{proof}

\begin{lemma}\label{lem:abstracting 1.4} Assume that $X$ is a group, $Y$ is a normal subgroup of $X$ and $x C_X(Y)\in Z(X/C_X(Y))$.  If $[Y,x] \le Z(Y)$, then $Y/C_Y(x) \cong [Y,x]$ as $X$-groups.
\end{lemma}

\begin{proof} Define \begin{eqnarray*}\theta :Y &\rightarrow &[Y,x]\\y &\mapsto& [y,x].\end{eqnarray*} Then $\theta$ is independent of the choice of  the coset representative in $xC_X(Y)$.

 For $y,z \in Y$, $$(yz)\theta=[yz,x]=[y,x]^{z} [z,x]=[y,x] [z,x]=(y)\theta (z)\theta,$$ and, for $y\in Y$ and $\ell  \in X$, as $[x,\ell]\in C_R(Y)$, $x^\ell= xc$ for some $c \in C_X(Y)$, and so  $$(y\theta)^\ell = [y,x]^\ell = [y^\ell, x^\ell]= [y^\ell,xc]=[y^\ell,c][y^\ell,x]^c=[y^\ell,x]=(y^\ell)\theta.$$  Thus $\theta$ is an $X$-invariant homomorphism from $Y$ to $[Y,x]$.  As $\ker\theta= C_{Y}(x)$, we have $Y/C_Y(x) \cong [Y,x]$ as $X$-groups.
\end{proof}

\begin{lemma}\label{lem:isochieffactors} Assume that $p$ is a prime, $X$ is a group, $Y$ is an abelian normal $p$-subgroup of $X$ and $R$ is a normal $p$-subgroup of $X$ which contains $Y$.  Suppose that $Y= [Y,O^p(X)]$, $[R,O^p(X)]\le C_R(Y)$ and $R$ acts quadratically or trivially on $Y$. Suppose that no non-central $X$-chief factor of $Y/C_Y(R)$ is isomorphic to an $X$-chief factor of $[Y,R]$.   Then $Y \le Z(R)$.
\end{lemma}

\begin{proof}  Assume that $R>C_R(Y)$. Using $[R,O^p(X)]\le C_R(Y)$, we may select $x \in R \setminus C_R(Y)$   such that $xC_X(Y) \in Z(X/C_X(Y))^\#$. As $Y$ is abelian, $[Y,x]\le Z(Y)$ and so \fref{lem:abstracting 1.4} applies to give $Y/C_Y(x) \cong [Y,x]$ as $X$-groups. As $R$ acts quadratically on $Y$, $$C_Y(x) \ge C_Y(R) \ge [Y,R] \ge [Y,x]$$ and so the hypothesis on non-central $X$-chief factors now gives $Y/C_Y(x)$ and $[Y,x]$ only have central $X$-chief factors. In particular, $Y= [Y,O^p(X)] \le C_Y(x)$ and this contradicts the initial choice of $x \in R \setminus C_R(Y)$. Hence $Y \le Z(R)$.
\end{proof}

\begin{lemma}\label{lem: U>YL abelian} Suppose that $p$ is a prime, $X$ is a group, $V \le U$  are  normal $p$-subgroups  of $X$, and $Q$ is a large $p$-subgroup of $X$ which is not normal in $X$.
Assume that $V$ is a non-trivial irreducible $\GF(p)X$-module and  $U/V $ is centralized by $O^p(X)$.Then \begin{enumerate}
\item $U$ is elementary abelian; and
\item if $U \not \le \Omega_1(Z(O_p(X)))$, then $O_p(X)/C_{O_p(X)}(U)$ contains a non-central chief factor isomorphic to $V$ as a $\GF(p)X$-module.
    \end{enumerate}
\end{lemma}

\begin{proof} Set $Z_X = \Omega_1(Z(O_p(X)))$. We have $[U,O^p(X)]\le V\le Z_X$ as $V$ is irreducible. As $O^p(X)$ does not centralize $U/\Phi(U)$ by Burnside's Lemma  \cite[Proposition 11.1]{GoLyS2} and $V$ is a non-trivial irreducible $X$-module, $V \not \le \Phi(U)$ and $\Phi(U)$ is centralized by $O^p(X)$.  Therefore $\Phi(U) \cap Z_X$ is centralized by $O^p(X)$ and is normalized by $Q$. Since $Q$ is large and $O^p(X) \not \le N_X(Q)$, we deduce $\Phi(U)\cap Z_X=1$. Thus $\Phi(U)=1$ and so $U$ is elementary abelian. Hence   (i) holds.

Set $Y=O_p(X)$ and assume that $U \not \le Z_X$. Select $x \in U\setminus Z_X$ such that $[X,x]\le  U\cap Z_X \le Z(Y)$.
Then $xC_X(Y) \in Z(X/C_X(Y))$.  Thus \fref{lem:abstracting 1.4} implies
$Y/C_{Y}(x) \cong [ Y,x]\le U \cap Z_X$ and this isomorphism is as $X$-groups.
 Since $[ Y,x]$  is normalized by $Q$,  $[Y,x] \ne 1$ and $Q$ is large, $O^p(X)$ does not centralize $[Y,x]$.  Thus $ [Y,x] \ge V$ as $[U,O^p(X)] \leq V$.  This proves (ii).
\end{proof}

\begin{lemma}\label{lem:dualchieffactors}
Assume that $p$ is a prime, $X$ is a group, $U$ is an elementary abelian normal subgroup of $X$, $U= [U,O^p(X)]$ and $O_p(X)$ acts quadratically and non-trivially on $U$. Set $R= O_p(X)$, $W= R/C_R(U)$, and $Z= [U,R]$.  Then $W$, $U/Z$ and $Z$ are $X/R$-modules and $W$ is isomorphic to an $X/R$-submodule of $\Hom(U/Z,Z)$.  In particular, if $Z$ is centralized by $X$, then  the set of $X$-chief factors of $W$ can be identified with a subset of the $\GF(p)$-duals of the $X$-chief factors of $U/Z$.
\end{lemma}

\begin{proof} Since $R$ acts quadratically on $U$, $W$ is elementary abelian. Furthermore, $R$ centralizes $W$, $U/Z$ and $Z$. Hence all of these groups can be regarded as $\GF(p)X/R$-modules. For $w \in R$, define \begin{eqnarray*}\theta: R &\rightarrow& \Hom(U/Z,Z)\\ w&\mapsto &  \begin{array}{rcl}\theta_w:U/Z&\rightarrow& Z\\uZ&\mapsto&[u,w]\end{array}.\end{eqnarray*}
The calculation in the proof of \fref{lem:abstracting 1.4} shows that the commutator $[u,w]$ defines a  homomorphism from $U$ to $Z$ and, as $w$ centralizes $Z$, $\theta_w$ is a well-defined homomorphism from $U/Z$ to $Z$. Thus $\theta$ is a well-defined map. Consider $w_1,w_2 \in R$, $uZ\in U/Z$  and $\ell \in X$. Then $$(uZ)\theta_{w_1w_2}= [u,w_1w_2]=[u,w_2]^{w_1}[u,w_1]=[u,w_1][u,w_2]= (uZ)\theta_{w_1}(uZ)\theta_{w_2}$$ which means $\theta_{w_1w_2}=\theta_{w_1}\theta_{w_2}$ and so $\theta$ is a group homomorphism.  We show that $\theta$ is an $X$-module homomorphism.  So let $\ell \in X$, $uZ\in U/Z$ and $w \in R$. Then
 $(w^\ell)\theta = \theta_{w^\ell}$ and $$(uZ)\theta_{w^\ell}= [u,w^\ell]= [u^{\ell^{-1}},w]^\ell= (u)(\theta_w\cdot \ell) .$$
Since $\ker \theta = C_R(U)$, this completes the proof of the main claim.

If $Z$ is centralized by $X$, then $$\Hom(U/Z,Z) \cong (U/Z)^*\otimes Z= \bigoplus_{i=1}^n (U/Z)^*$$ where $n$ is such that $|Z|=p^n$. This completes the proof of the lemma.
\end{proof}

\begin{lemma}\label{lem:soldual} Suppose that $V$ is a $p$-group and $X$ is a group which acts faithfully on $V$ with $O_p(X)=1$. Assume $A \leq X$ is an elementary abelian $p$-subgroup  of  order at least $p^2$
which has the property  $C_V(A) = C_V(a)$ for all $a \in A^\#$. If $L$ is a non-trivial subgroup of $X$ and $L=[L,A]$,  then $A$ acts faithfully on $L$.

In particular, $A$ centralizes every $p'$-subgroup which it normalizes,
 $[A,F(X)] = 1$, $E(X) \ne 1$ and, if $L$ is a component of $X$ which is normalized but not centralized by  $A$,  then $A$ acts faithfully on $L$.
\end{lemma}

\begin{proof}
 Suppose that $L=[L,A]$  is a non-trivial subgroup of  $X$. Assume that there is $b \in A^\#$ with $[L,b] = 1$. Then  $L$ normalizes $C_V(b)$ and so, as $C_V(b)=C_V(A)$, $L=[L,A]$  centralizes $C_V(b)$. Since $L= [L,A]$, $L= O^p(L)$ and the Thompson $A \times B$-Lemma   implies $[L,V] = 1$, a contradiction. Hence $A$ acts faithfully on $L$.

 Let $F$ be a $p'$-subgroup of $X$ which is normalized by $A$. Then $F= \langle C_F(a)\mid a \in A^\#\rangle$. If $A$ does not centralizes $F$, then there exists $a \in A^\#$ such that $1 \ne [C_F(a),A] = [C_F(a),A, A ]$.  Hence, taking $L= [C_F(a),A]$, we have $L= [L,A]$ and $a \in C_A(L)$, a contradiction.  Hence $[F,A] =1$.  Now $A$ centralizes $F(X)$ and therefore $E(X) \ne 1$.

 If $L$ is a component of $X$ which is normalized by $A$, then either $[L,A]=L$ or $[L,A]=1$. If $[L,A] \ne 1$, then we have $A$ acts faithfully on $L$.
\end{proof}

\begin{lemma}\label{lem:prank} Let $X$ be a group, $N$ a normal subgroup of $G$ and $T\in \syl_p(X)$. Assume that $X=NT$, $C_T(N)=1$, $q = p^{a}$ and
$$N=N_1 \times N_2 \cdots \times N_s,$$ where $N_i \cong \SL_2(q)$ for $1\le i\le s$. Then the $p$-rank of $G$ is $sa$.
\end{lemma}

\begin{proof} Assume first that $q = 2$. Then $T$ acts faithfully on $O_3(N)$.  As the $2$-rank of $\GL_s(3)$ is $s$, we are done. Similarly, if $q = 3$, then $T$ acts faithfully on $O_2(N)/Z(N)$, which is elementary abelian of order $2^{2s}$ we are done as $\GL_{2s}(2)$ has $3$-rank $s$.

Thus we may assume that $q >3$. In particular, the subgroups  $N_i$ are quasisimple and  $T$ permutes the set $\{N_i\mid 1\le i \le s\}$.

Assume   that $p$ is odd. Let $A$ be an elementary abelian subgroup in $T$ of maximal rank and assume that $A \not\leq N$. Then by Thompson replacement \cite[Theorem 25.2]{GoLyS2} we may assume that $A$ acts quadratically on $T \cap N$. This shows that $A$ has to normalize each $N_i$. As non-trivial field automorphisms are not quadratic on $T \cap N_i$, we get that $A$ centralizes $T \cap N$ and so $A \leq T \cap N$, the assertion.

Assume that $q = 2^{a}$ with $a \geq 2$. Let $B = N_N(T \cap N)$. We have that $T$ normalizes $B$ and $T/(T \cap N)$ acts faithfully on $B/(T \cap N)$. Now   the Thompson dihedral Lemma \cite[Lemma 24.1]{GoLyS2} says that for any elementary abelian subgroup $A$ of $T$ we have a $B$-conjugate $A^g$ such that $U = \langle A, A^g \rangle (T \cap N)/(T \cap N)$ is a direct product of $r $ dihedral groups where $2^r = |A/(A \cap N)| \leq 2^s$ and $A(T \cap N)/(T \cap N)$ is a Sylow 2-subgroup of $U$. Set $T_1 = [O_{2^\prime}(U),T \cap N]$. As $U$ is generated by two conjugates of $A$ we see that $|T_1| = |C_{T_1}(A/A \cap N)|^2$. This now shows that $|A| \leq |T \cap N|$, the assertion again. This proves the lemma.
\end{proof}

In the next two lemmas we use the notation presented in the  introduction though we do not assume that $L$ is unambiguous.

\begin{lemma}\label{lem:VL sub QL'} Suppose that $L \in \mathcal{L}_G(S)$, $L \not\le N_G(Q)$   and $V_L = [Y_L,L^\circ]$. Then
\begin{itemize}\item[(i)] $C_{Y_L}(L^\circ) = 1$.
\item [(ii)]    $\Omega_1(Z(S)) \le V_L$.
\item[(iii)] If $V_L$ is an irreducible $L$-module, $V_L \not \le Q$  and $\Omega_1(Z(Q_L)) < Q_L$, then $V_L \leq Q_L^\prime \leq \Phi (Q_L)$.
\end{itemize}
\end{lemma}
\begin{proof} As $C_{Y_L}(L^\circ) \leq C_G(Q)$ is normalized by $L$,   (i) is a consequence of $Q$ being large.

By \cite[Lemma 1.24 (g)]{stru},  $\Omega_1(Z(S)) \le Y_L$ now Gasch\"utz Theorem  \cite[Theorem 9.26]{GoLyS2} and (i) give (ii).

     Assume that $N$ is a non-trivial normal $p$-subgroup of $L$. Then $\Omega_1(Z(S)) \cap N \ne 1$.  Since $V_L$ is irreducible as a $L$-module, (ii) gives $V_L \le N$. In particular, as $V_L \not\le Q$, $N \not \le Q$.

      Suppose that $Q_L$ is abelian. Then, as $Q=O_p(N_G(Q))$ and $[Q,Q_L,Q_L]\le Q_L'=1$, $Q_L$ is quadratic on $Q$, and hence $Q_LQ/Q$ is elementary abelian and so $\Phi(Q_L) \le Q$. By the remark earlier taking $N = \Phi(Q_L)$ we obtain $\Phi(Q_L) =1$, contrary to $\Omega_1(Z(Q_L)) < Q_L$.  Hence $Q_L$ is non-abelian. Thus  $Q_L' \ne 1$ and so, as $V_L$ is irreducible, $V_L \le Q_L' \le \Phi(Q_L)$. This proves (iii).
\end{proof}

\begin{lemma}\label{lem:[WO_p][W^mO_p]}  Suppose that $L \in \mathcal{L}_G(S)$, $L \not\le N_G(Q)$  and $V_L = [Y_L,L^\circ]$. Assume that $Y_L = \Omega_1(Z(Q_L))$, $m \in L$ and $O^p(L)Q_L \le KQ_L$, where $K = \langle W, W^m\rangle$. Then $O^p(L)\le K$ and the following hold
\begin{enumerate} \item[(i)] $[O^p(L),Q_L ]\le [W,Q_L][W^m,Q_L] \le (W\cap Q_L)(W^m\cap Q_L)=U_L .$
\item [(ii)]If  $[W,W] \le V_L$, then $W$ acts quadratically on the non-central chief factors of $Q_L/V_L$.
\end{enumerate}
 Assume, in addition, that $V_L$ is irreducible as a $K$-module,  $[V_L,W,W]\ne 1$,
and $[W,W] \le V_L$.  Then
\begin{enumerate}
\item [(iii)] $W \cap W^m \cap Q_L\le Y_L$;
\item [(iv)]  $U_L/Y_L$ is elementary abelian or trivial;
and
\item[(v)] either  $Q_L = Y_L$ or  $U_L' \ge V_L$.
\end{enumerate}
\end{lemma}

\begin{proof} Since $W$ and $W^m$ are normalized by $Q_L$, $K=\langle W,W^m\rangle$ is normalized by $Q_LK$ and so $O^p(L) \le K$.
 Since $W$, $W^m$, $[Q_L,W]$ and $[Q_L,W^m]$ are normalized by $Q_L$, we have $$[Q_L,O^p(L)] \le [Q_L, \langle W, W^m\rangle]=[Q_L, W][Q_L,  W^m] \le (W\cap Q_L)(W^m\cap Q_L).$$
 In particular, $A = (W\cap Q_L)(W^m\cap Q_L)$ is normalized by $O^p(L)$. Since $(W \cap Q_L)^L = (W \cap Q_L)^{SO^p(L)}= (W\cap Q_L)^{O^p(L)}$, we have $A = U_L$.
 Thus (i) holds.

By the  additional hypothesis,  $$[Q_L,W,W]\le [W,W]\le V_L$$ and so $W$ acts quadratically on all the non-central $L$-chief factors  in $Q_L/V_L$, which is (ii).

Notice that part (ii), $V_L$ irreducible as a $K$-module and $[V_L,W,W]\ne 1$  together imply that the non-central $K$-chief factors in $Q_L/V_L$ are not isomorphic to $V_L$.

Set $I= W \cap W^m \cap Q_L$. Then $I \le W \cap W^m$ and so $$[I,W] \le [W,W]\le V_L$$ and $$[I,W^m]\le [W^m,W^m]\le V_L^m= V_L.$$  Hence $IV_L/V_L$ is centralized by $\langle W,W^m\rangle = K$.
   As $W$ acts quadratically on all the non-central chief factors of $K$  in $Q_L/V_L$  by (ii) and by assumption, $W$ does not act quadratically on $V_L$, \fref{lem: U>YL abelian} implies that $I \le \Omega_1(Z(Q_L))=Y_L$. This proves (iii).

 Since $W$ is generated by elements of order $p$, $W/[W,W]$ is elementary   abelian and therefore, as   $[W,W]\le V_L$, $WV_L/V_L$ is also elementary abelian. Since $W \cap Q_L$ and $Q_L\cap W^m$ normalize each other parts (i) and (iii) give    (iv).

If $V_L \not \le U_L' $  and $Q_L \not= Y_L$, then, as $U_L/Y_L$ is elementary abelian by (iv), \fref{lem:VL sub QL'} (ii) implies $U_L$ is elementary abelian. Select $E$ with $Q_L\ge E > V_L$ of minimal order such that $E= [E,O^p(L)]$ and $E/V_L$ has a  non-central $K$-chief factor.   Then $$E\le  [Q_L,O^p(L)]\le [Q_L,W][Q_L,W^m]\le U_L\le C_{L}(E).$$ Furthermore, $V_L[E,Q_L]< E$ and so $[[E,Q_L], O^p(L)]\le V_L$. Therefore \fref{lem: U>YL abelian}  implies that $[E,Q_L]\le Y_L$ and so $Q_L$ acts quadratically on $E$.
 Hence \fref{lem:isochieffactors} implies that $E \le Y_L$, a contradiction. Hence $U_L'$ is non-trivial and it follows that $V_L \le U_L'$.
\end{proof}

\section{The reduction}\label{sec:wreath cases}

We use the notation presented in the introduction. For the rest of this article we have $L \in \mathcal{L}_G(S)$ with $Q$  not normal in $L$ and $L$  is in the unambiguous wreath product case.
This means that $Y_L= V_L$ unless we are in the special case that $\ov{L^\circ} \cong \SL_2(4)$ or $\GammaSL_2(4)$, $|Y_L:V_L|=2$ and $$V_L \not \le Q.$$

We start with a general result which just requires $V_L\not\le Q$.

\begin{lemma}\label{lem:w1} The following hold.
\begin{enumerate}
\item $\langle V_L^{D}\rangle$ is not a $p$-group;
\item $[Q,\langle V_L^{D}\rangle ]\le W$; and
\item $W \not \le C_G(V_L)$.
\end{enumerate}
\end{lemma}

\begin{proof} Let $\tilde C= N_G(Q)$ and $K = \langle V_L^{\tilde C}\rangle$. As $D = KN_L(Q)$ and $N_L(Q)$ acts on $V_L$ we have  $\langle V_L^{D}\rangle = \langle V_L^K \rangle$  is subnormal in $H$. If $\langle V_L^D \rangle$ is a $p$-group, we obtain $V_L  \le O_p(N_G(Q))=Q$ which is  a contradiction. This proves (i).

 We have $[Q,V_L] \le Q \cap V_L \le W$.  As $W$ and $Q$ are normalized by $D$, (ii) holds.

Assume  $W \le C_G(V_L)$. Then $[W,V_L]=1$ and so $[W,\langle V_L^{D}\rangle]=1$. Hence $X=O^p(\langle V_L^{D}\rangle)$ centralizes $Q$ by  (ii). Since $C_G(Q) \le Q$, we have  $X\le Q$.   Thus $X=1$  and $\langle V_L^{D}\rangle$ is a $p$-group, which contradicts (i). Hence $W \not \le C_G(V_L)$.
\end{proof}

 We adopt the following  notation.
Let $B \ge C_L(V_L)$ be such that  $\ov B= \langle  \mathcal K\rangle$ and let $S_0 = S \cap B$.
We write $B= K_1 \dots K_s$ where $K_i \ge C_L(V_L)$, $\ov{K_i}\in \mathcal K$, $\ov {K_i} \cong \SL_2(q)$ and, for $1 \le i \le s$, put $$S_i= S \cap K_i$$ $$V_L^{i}= [V_L,K_i],$$ $$Z_{i}=  C_{V_L^i}(S_i)=C_{V_L^i}(S_0)$$ and
$$Z_0 = Z_1 \dots Z_s= C_{V_L}(S_0).$$

We begin by showing that $\ov W$ is not contained in the base group $\ov B$.

\begin{lemma}\label{lem:not base} Suppose that $\ov L$ is either properly wreathed,   or  $q= p^a$ (where $p$ divides $a$) and   some element of $\ov{ L^\circ}$ induces a non-trivial field automorphism on  $O^p( \ov{L^\circ}) \cong \SL_2(q)$. Then $W $ is not contained in $S_0$. In particular, if $\ov L$ is properly wreathed with $q=s=2$,   then $\ov Q$ is not cyclic of order $4$.
\end{lemma}
\begin{proof} Set $F=  \bigcap_{g \in D} C_Q(V_L)^g .$

 Suppose that $W $ is  contained in $S_0$.  As $\ov Q$ normalizes $\ov W$ and acts transitively  on $\mathcal K$ when $\ov L$ is properly wreathed  and,  as  $V_L$ is the natural $\SL_2(q)$-module when $s=1$, and field automorphisms are present, the structure of $V_L$ yields
 $$[V_L,S_0]=[V_L,W]=C_{V_L}(W) = Z_0.$$
Suppose that $g \in D$. Then using \fref{lem:w1}(ii)  and $(V_L)^g =V_{L^g}$ yields

\begin{claim}\label{clm:notbase1}$[Z_0,[V_{L^g},Q]] \le [Z_0,W]=1.$\end{claim}

\medskip
 We also remark that as $W \le Q$, $Z_0 \le [V_L, Q]\le W = W^g\le S_0^g$ and $Z_0 \le Z(W)$. In particular, as $S_0^g$
 normalizes every element of $\mathcal K^g$, so does
 $Z_0$.  Therefore, for $1\le i \le s$, $Z_0$   also normalizes each $K_i^g$ and so also $[Y_L^g, K_i^g ]= (V_{L}^i)^g$.

   If $s=1$ and we have field automorphisms in $\ov{L^\circ}$, then $[V_L,Q] > Z_0$ and so     \fref{clm:notbase1} provides $Z_0\le C_Q([V_{L^g},Q]) = C_Q(V_{L^g})$. Thus $$[V_L,W]=  Z_0 \leq F$$ in this case.

   We will show that the same holds in the properly wreathed case.
Because $Q$ acts transitively on $\mathcal K ^g$,
 $$V_{L^g}= V_{L^g}^1 [V_{L^g},Q] =V_{L^g}^2[V_{L^g},Q].$$ As  $[Z_0,[V_{L^g},Q]]=1$ by \fref{clm:notbase1},
 \begin{eqnarray*}[V_{L^g},Z_0] &=& [V_{L^g}^1 [V_{L^g},Q], Z_0]\cap [V_{L^g}^2[V_{L^g},Q], Z_0] \\&=&[ V_{L^g}^1 , Z_0]\cap [V_{L^g}^2, Z_0] \le V_{L^g}^1\cap V_{L^g}^2=1.\end{eqnarray*} Hence $Z_0 \le C_{Q}(V_{L^g})$ and this implies that $$[V_L,W]=Z_0 \le F$$ in the properly wreathed case too.
 Therefore,
\begin{eqnarray*} [Q,V_L] &\le& W \cr [W,V_L]&=&Z_0 \le F\cap W\cr [F\cap W,V_L]&=&1. \end{eqnarray*}
Hence $V_L$ stabilizes the normal series $Q\ge W\ge W\cap F \ge 1$ in $D$ and so  $V_L \le O_p(D)$. But then $\langle V_L^D\rangle $ is a $p$-group contrary to \fref{lem:w1} (i).  We conclude that $W \not \le {S_0}$ as claimed.

If $q=s=2$  and $\ov Q $   is cyclic of order four, then, as $\ov W$ is generated by involutions, $\ov W= \ov Q\cap \ov S_0$, a contradiction. Thus $\ov Q $ is not cyclic of order $4$ in this case.
\end{proof}
We now reduce the properly wreathed case to one specific configuration which will be handled in detail in \fref{sec:o4}.

\begin{proposition}\label{prop:notwreath} Assume that $\ov L$ is properly wreathed and  unambiguous. Then $|\mathcal K|=2$, $q=2$, and $\ov W$ permutes $\mathcal K$ transitively by conjugation. Furthermore,   $\ov Q = \ov S\cong \Dih(8)$, $\ov {L^\circ} \cong \mathrm O_4^+(2)$ and $Y_L=V_L$ is the natural $\mathrm O_4^+(2)$-module.
\end{proposition}

\begin{proof}
Since $Q$ permutes $\mathcal K$ transitively by conjugation and $S_0$ normalizes $Q$,   we have

\begin{claim}\label{clm:index [Q,S_0]}\begin{enumerate}
\item $\ov {Q \cap S_0}$ contains $[\ov Q,\ov {S_0}]$;
 \item$|\ov{S_0}:\ov {Q\cap S_0}| \le |\ov{S_0}:[\ov Q,\ov S_0]|\le q;$ and
 \item $\ov{[Q,S_0]}C_{\ov L}(\ov{K_i})/C_{\ov L}(\ov{K_i})\in \syl_p(\ov{K_i}C_{\ov L}(\ov{K_i})/C_{\ov L}(\ov{K_i}))$.\qedc\end{enumerate}\end{claim}
\medskip

As $W=\langle V_{L^g}\cap Q\mid g \in D\rangle$, \fref{lem:not base} implies  there exists $g \in D$ such that $V_{L^g}\cap Q \not \le S_0$. We fix this $g$.

\begin{claim}\label{clm:meet base} We have $\ov {V_{L^g} \cap Q} \cap \ov {S_0}\ne 1$.\end{claim}

\medskip

Suppose that $\ov {V_{L^g} \cap Q }\cap \ov {S_0}= 1$. Then, as $\ov {Q \cap S_0}$ and  $\ov {V_{L^g} \cap Q}$ normalize each other, $\ov {V_{L^g} \cap Q}$ centralizes $\ov {Q \cap S_0}$.
If $\ov {V_{L^g} \cap Q}$ normalizes  some $\ov{K_i}\in \mathcal K$, then, as $\ov Q$ acts transitively on $\mathcal K$ and normalizes $\ov {V_{L^g} \cap Q}$, $\ov {V_{L^g} \cap Q}$ normalizes every member of $\mathcal K$.
   As $\ov {V_{L^g} \cap Q}$ centralizes $\ov{[Q,S_0]}$, \ref{clm:index [Q,S_0]} (iii) implies that $$\ov{ V_{L^g} \cap Q} \le \ov{[Q,S_0]}C_{\ov L}(\ov{K_i}).$$ Since $Q$ acts transitively on $\mathcal K$, this is true for each $\ov{K_i} \in \mathcal K$.
 Thus $$\ov{V_{L^g} \cap Q} \le \bigcap_{i=1}^s  \ov{[Q,S_0]}C_{\ov L}(\ov{K_i}) = \bigcap_{i=1}^s  \ov{S_i}C_{\ov L}(\ov{K_i}) =\ov{S_0},$$ which contradicts the choice of $g\in D$.

 Hence $\ov { V_{L^g} \cap Q}$ does not normalize any member of $\mathcal K$. As $\ov B$ is a direct product we calculate that $C_{\ov S_0}(\ov {V_{L^g} \cap Q})$ has index at least $q^{p-1}$ in $\ov{S_0}$. However \fref{clm:index [Q,S_0]} (ii) states that $\ov {Q \cap S_0}$ has index at most $q$ in $\ov{S_0}$  and, as this subgroup  is centralized by $\ov{V_{L^g}\cap Q}$, we deduce that $$p=2.$$ Furthermore,  as $\ov {V_{L^g} \cap Q}$ does not normalize any member of $\mathcal K$, if $s> 2$, we have   $C_{\ov S_0}(\ov {V_{L^g} \cap Q})$ has index at least $q^{2} $ in $\ov{S_0}$, and so we must have $$s=2.$$

 Since $\ov {V_{L^g} \cap Q}$ centralizes $[\ov{S_0},\ov{Q}]$ by \fref{clm:index [Q,S_0]}(iii),  no element in  $\ov {V_{L^g} \cap Q}$ can act as a non-trivial field automorphism on $\ov {K_1}$ and so  we   infer from $\ov {V_{L^g} \cap Q} \cap \ov {S_0}= 1$, that  $|\ov {V_{L^g} \cap Q}|=2$. In particular, $|C_{V_L}(V_{L^g} \cap Q)|= q^2$ as $V_{L^g} \cap Q$ exchanges $V_L^1$ and $V_L^2$.

We know that $|V_{L^g}|= q^4$. As  $|[V_{L^g},Q]|\ge q^3$, we have $$|V_{L^g}:V_{L^g}\cap Q|\le q,$$ and we have just determined that  $$|V_{L^g} \cap Q:V_{L^g} \cap Q\cap C_G(V_L)|=|\ov {V_{L^g} \cap Q}|=2.$$
Hence $V_{L^g} \cap Q \cap C_G(V_L)  $ has order at least $2^{3a-1}$, where $q=2^a$.

Assume that $a \ne 1$.  Then, as $V_{L^g}^1$ has order $q^2$, $$V_{L^g} \cap Q \cap C_G(V_L) \cap V_{L^g}^1\ne 1.$$
It follows that $V_L\cap Q $  normalizes both $K_1^g$ and $K_2^g$. As $(V_L\cap Q)C_{L^g}(V_{L^g})/C_{L^g}(V_{L^g})$ is normalized by $Q$ and $Q$ permutes $\{K_1^g,K_2^g\}$ transitively, $(V_L\cap Q)C_{L^g}(V_{L^g})/C_{L^g}(V_{L^g})$ does not centralize $K_i^g /C_{L^g}(V_{L^g})$ for $i=1,2$. Thus $|C_{V_{L^g}^i}(V_L\cap Q)|\le q$ for $i=1,2$. But then $$2^{3a-1}\le |V_{L^g} \cap Q \cap C_G(V_L) |\le | C_{V_{L^g}}(V_L\cap Q)| \le 2^{2a},$$ which contradicts $a \ne 1$. We conclude that    $q=s=2$ and $|\ov {V_{L^g} \cap Q}|=2$. Furthermore, $\ov{V_{L^g}\cap Q}$ is centralized by $\ov {Q}$ and so $\ov Q$ is elementary abelian of order $4$.  It follows that $\ov{L^\circ}
\cong \Omega_4^+(2)$ and $V_L$ is the natural module.   Hence  $L$ is  ambiguous and we conclude that  $\ov {V_{L^g} \cap Q} \cap \ov {S_0} \ne 1$.\qedc

\begin{claim}\label{clm:cent bound}  We have
$|C_{V_L}(V_{L^g} \cap Q)| \le q^{s/p}.$
\end{claim}

\medskip

We know $\ov {V_{L^g} \cap Q} \not \le \ov{S_0}$ and $\ov {V_{L^g} \cap Q} \cap \ov S_0\ne 1$ by \fref{clm:meet base}. As $\ov {V_{L^g} \cap Q}$ is normalized by $\ov Q$, $\ov {V_{L^g} \cap Q} \cap \ov S_0\ne 1$ implies that $$ C_{V_L}(\ov { V_{L^g} \cap Q}) =C_{Z_0}(\ov {V_{L^g} \cap Q}).$$
If some element  $d \in V_{L^g} \cap Q$ induces a non-trivial field automorphism on $\ov K_i$ for some $\ov{K_i}\in \mathcal K$, then
$C_{V_L^i}(  {V_{L^g} \cap Q}) \le C_{Z_i}(d)$ has order $q^{1/p}$ and the result follows by transitivity of $\ov Q$ on $\mathcal K$. On the other hand, if  $d \in  V_{L^g} \cap Q$ has an orbit of length $p$ on $\mathcal K$, then $C_{\langle (V_L^1)^{\langle d \rangle}\rangle}( {V_{L^g} \cap Q}) \le C_{\langle Z_1^{\langle d \rangle}\rangle}(d)$ which has order $q$.  Using the transitivity of $Q$ on $\mathcal K$, we deduce $|C_{V_L}(V_{L^g} \cap Q)| \le q^{s/p}$. This proves the result.
\qedc

As $Q$ acts transitively on the $\{V_i\mid 1\le i \le s\}$, we have $V_L = [V_L,Q]V_1$. By \fref{clm:meet base} $\ov{Q} \cap \ov{S_0} \not= 1$ and so $|[V_1,Q]| \geq q$. In particular
$$|V_L:[V_L,Q]|\le q.$$

Since  $V_L\cap Q \cap C_{L^g}(V_{L^g}) \le  C_{V_L}(V_{L^g} \cap Q)$, \fref{clm:cent bound} and $|V_L|=q^{2s}$ together give $$|(V_L\cap Q)C_{L^g}(V_{L^g})/C_{L^g}(V_{L^g})| \ge q^{2s-1-s/p}.$$
On the other hand,  by \fref{lem:prank} the $p$-rank of $\ov L$ is $as$ where $q=p^a$.  Hence $$s \ge 2s-1-s/p$$ and so $$s=p=2.$$
  In particular,   \fref{lem:prank} implies

\begin{claim}\label{clm:rank}   $|(V_L\cap Q)C_{L^g}(V_{L^g})/C_{L^g}(V_{L^g})| = q^{2}=2^{2a}$.\end{claim}

\medskip

Assume that $q> 2$. Since $S^g/S_0^g$ has $2$-rank $2$ and $V_L\cap Q$ is elementary abelian,   $(V_L \cap Q \cap S_0^g)C_{L^g}(V_{L^g})/C_{L^g}(V_{L^g})$ has rank at least $2a-2\ne 1$.
 Since $V_L \cap Q \cap S_0^g$ is normalized by   $Q$ and $Q$ permutes $\{K_1^g,K_2^g\}$ transitively, $ V_L \cap Q \cap S_0^g$ contains an element which projects non-trivially on to both $S_1^gC_{L^g}(V_{L^g})/C_{L^g}(V_{L^g})$ and $S_2^gC_{L^g}(V_{L^g})/C_{L^g}(V_{L^g})$. 
Thus $V_L\ge [V_L\cap Q, [V_{L^g},Q]]\ge Z_0^g$. But then, using \fref{clm:cent bound} yields the contradiction  $$q^2= |Z_0^g| \le |C_{V_L}(V_{L^g} \cap Q)| \le q.$$ Thus $q=s=2$.
 It follows from \fref{lem:not base} that $W$ is transitive on $\mathcal{K}$ and $\ov Q \cong \Dih(8)$ or $\ov Q$ is elementary abelian of order $4$.  The second possibility gives $\ov{L^\circ} \cong \Omega^+_4(2)$, which is ambiguous. This proves \fref{prop:notwreath}.
\end{proof}

Next we deal with the case $s= 1$.

\begin{proposition}\label{prop:fieldauto} Suppose that $O^p( \ov{ L^\circ}) \cong \SL_2(q)$ where  $q= p^a= r^p$, $V_L = Y_L$ is the natural $O^p( \ov {L^\circ})$-module and that some element of $\ov {L^\circ}$ induces a non-trivial field automorphism on  $O^p( \ov {L^\circ})$. Then  $p =2 = r$.
\end{proposition}

\begin{proof} We may assume that $r^p > 4$.  By \fref{lem:not base} we have that $W \not\leq S_0$ and, as $W$ is generated by elements of order $p$, we have that $|S_0W : S_0| = p$.  As $Q$ is normal in $S$, $1 \not = \ov{Q} \cap \ov{S}_0$, so $Z_0 \leq Q \cap Y_L$. Furthermore, as $\ov Q$ contains elements which act as field automorphisms on $O^p(\ov {L^\circ})$,  $$|V_L \cap Q :Z_0| \ge |[V_L,Q]:Z_0|\geq r^{p-1}> p,$$ by assumption. Thus no element in $S \setminus Q_L$ centralizes a subgroup of index $p$ in $V_L \cap Q$.

Set $W_1 = \langle Z_0^{D} \rangle$. As $Z_0$ centralizes $W \cap S_0$, every element of $Z_0$ centralizes a subgroup of index at most $p$ in $W$. As $W_1$ is generated by conjugates of $Z_0$, and these conjugates all contain elements which centralize a subgroup of index at most $p$ in $W$, $W_1$ is generated by elements which centralize a subgroup of index at most $p$ in  $V_L \cap Q$.  As no element in $S \setminus Q_L$ has this property, we conclude that $W_1 \leq Q_L$. Hence  $[V_L, W_1] = 1$. In particular $[V_L \cap Q, W_1] = 1$ and so also $[W, Z_0]=[W,W_1] = 1$.  This shows $W \leq S_0$ and contradicts \fref{lem:not base}.
\end{proof}

We collect the results of this section in the following proposition:

\begin{proposition}\label{prop:unambiguous} Suppose that $L \in \mathcal{L}_G(S)$, $L \not \le N_G(Q)$, $V_L \not\leq Q$ and $ L$ is in the unambiguous wreath product case. Then one of the following holds:
\begin{itemize}
\item[(i)] $\ov {L^\circ} \cong \mathrm O_4^+(2)$ ,   $\ov Q = \ov S \cong \Dih(8)$ and $Y_L=V_L$ is the natural module.
\item[(ii)]  $\ov {L^\circ} \cong \Gamma \SL_2(4)$, $V_L$ is the natural $\SL_2(4)$-module and $|Y_L : V_L| \leq 2$.
\item[(iii)] $\ov {L^\circ} \cong \SL_2(4)$, $V_L$ is the natural module and $|Y_L : V_L| = 2$.
\end{itemize}
\end{proposition}

\begin{proof} If $|\mathcal{K}| > 1$, then (i) holds by \fref{prop:notwreath}, so we may assume that  $|\mathcal{K}| = 1$. As $L$ is unambiguous, either $Y_L \not= V_L$ or $\ov{L^\circ} \not\cong \SL_2(q)$. If $Y_L \not= V_L$, then by definition of the wreath product case, (ii) or (iii) holds. So we may assume $Y_L = V_L$ and $\ov{L^\circ} \not\cong \SL_2(q)$. Now (ii) holds by  \fref{prop:fieldauto}.
\end{proof}

\section{$\ov {L^\circ} \cong \OO^+_4(2)$}\label{sec:o4}

In this section we analyse the configuration  from \fref{prop:unambiguous}(i).  We prove

\begin{proposition}\label{prop:O4} Suppose that  $L \in \mathcal{L}_G(S)$, $L \not \le N_G(Q)$  and $ L$ in the unambiguous wreath product case. If  $Y_L \not\leq Q$ and $\ov {L^\circ} \cong \OO^+_4(2)$, then $G \cong \Sym(8)$, $\Sym(9)$ or $\Alt(10)$.
\end{proposition}

\begin{proof}
 By \fref{prop:unambiguous} we have  $\ov {Q}\cong \Dih(8)$.  Since $Y_L$ is the natural $\OO_4^+(2)$-module for $L/C_L(Y_L)$ and $V_L$ is also the wreath product module for $L/C_L(Y_L)$ with respect to $\{\ov{K_1}, \ov {K_2}\}$, we have the following well known facts.

 \begin{claim}\label{clm:YLprops} \begin{enumerate}
 \item $|[Y_L, Q]| = 2^3$, $|[Y_L,Q,Q]|= 2^2$ and $C_{Y_L}(Q)=[Y_L,Q,Q,Q]$ has order $2$.
 \item $[Y_L,S_0]= C_{Y_L}(S_0)$ has order $2^2$;
 \item $|[Y_L, Q']|= 2^2$;
 \item $C_L([Y_L, Q]) \le C_L(Y_L)$.
 \end{enumerate}
 \end{claim}

  Our first aim is to prove

\begin{claim}\label{clm:W4} $\ov W$ is elementary abelian of order $2^2$, $[Y_L,W]= [Y_L,Q]= Y_L\cap Q$  and $[Y_L,W,W]=  C_{Y_L}(W)= C_{Y_L}(Q)=Z$.
\end{claim}

\medskip

Applying \fref{lem:w1}, we consider $x \in D$ such that $Y_{L^x}\cap Q \not \le C_L(Y_L)$. Then $Y_{L^x}\cap Q$ is normalized by $Q$ and so
$$\ov {Y_{L^x} \cap Q}\text{ contains a 2-central involution in }\ov Q.$$ In particular,  \fref{clm:YLprops}(iii) gives $$|[Y_L,Y_{L^x} \cap Q]| \geq 2^2.$$  As $Y_L$ is elementary abelian, $\ov {Y_{L^x} \cap Q}$ is elementary abelian.

Suppose that $[Y_L, Y_{L^x}\cap Q, Y_{L^x} \cap Q] =1$. Then $$[Y_L,Y_{L^x}\cap Q]\le C_{S^x}([Y_{L^x},Q])= Q_{L^x}$$ by \fref{clm:YLprops} (iv).  Hence $[Y_L, Y_{L^x}\cap Q, Y_{L^x}]=1$.  Then as $|[Y_L,Y_{L^x}\cap Q]|=2^2$ and $|Y_L \cap Q|=2^3$, we conclude that $(Y_L \cap Q)C_{{L^x}}(Y_{{L^x}})/C_{{L^x}}(Y_{{L^x}})$ has order $2$.  Thus $[Y_{{L^x}}, Y_L\cap Q, Y_L \cap Q]=1$.  Now the argument just presented implies that $|\ov{Y_{L^x}\cap Q}| = 2$ and so,  as $Q$ normalizes $Y_{L^x}\cap Q$, $\ov{Y_{L^x}\cap Q} =Z(\ov Q)$.
In particular, as $[Y_L,S_0,S_0] =1$, we have proved that

 \begin{center}if $\ov{Y_{L^x}\cap Q} \le \ov {S_0}$, then $\ov{Y_{L^x}\cap Q}= Z(\ov Q)$. \end{center}

\medskip
 For a moment let $\ov{Q_1}$ be the fours subgroups of $\ov Q$ not equal to $\ov {S_0}$.
  Then  as $\Phi (Y_{L^x} \cap Q) = 1$ the displayed line implies that $\ov W \le \ov{Q_1}$ and  \fref{lem:not base}  and  $\ov{Q}^\prime \leq \ov{Y_{L^x} \cap Q}$  imply $\ov{W} = \ov {Q_1}$. The remaining statements in \fref{clm:W4} now follow from the action of $L$ on $Y_L$.\qedc

We have that $Z(Q)$ centralizes $[Y_L,Q] $ and so $Z(Q) \le S\cap C_L(Y_L)= Q_L$. Hence
using \fref{clm:W4} we obtain  \begin{eqnarray*}[W,W]&=&[\langle [Y_L,Q]^{D}\rangle ,W]= \langle [[Y_L,Q] ,W]^{D}\rangle\\&= &\langle Z^{D}\rangle= Z[Z,\langle V_L^{N_G(Q)}\rangle    ] \le Z[Z(Q),\langle V_L^{N_G(Q)}\rangle]\\& =&Z\langle [Z(Q),V_L]^{N_G(Q)}\rangle=Z.\end{eqnarray*}
\begin{claim}\label{clm:qlyl}
We have $Q_L= Y_L$.
\end{claim}

\medskip

Suppose that $Q_L > Y_L$. Let $m \in L$  be such that $\ov{K} \cong \SL_2(2)\times \SL_2(2)$, where $K =  \langle W,W^m\rangle$. Recall that by the choice of $L$ in the Notation at the end of  the introduction, we have $Y_L=\Omega_1(Z(Q_L))$ and by \fref{prop:unambiguous} and  \fref{clm:W4}, $K$ acts irreducibly on $Y_L=V_L$.  Hence  we may apply
 \fref{lem:[WO_p][W^mO_p]} (iii), (iv) and (v) which combined yield $U_L/Y_L$ is elementary abelian and  $$U_L'=Y_L.$$

 Since $[Q_L,W,W] \le [W,W] =Z \leq Y_L$, we have $W$ acts quadratically on every chief factor of $L$ in $Q_L/Y_L$.  In particular, no non-central $L$-chief factor of $Q_L/Y_L$ is isomorphic to $Y_L$.

Let $E$ be the preimage of $C_{U_L/Y_L}(K)$.  Then $E$ is normal in $L$ and application of \fref{lem: U>YL abelian}  implies that $E= Y_L$. Let $X \in \Syl_3(K)$. By \fref{lem:[WO_p][W^mO_p]}(i), $[K,C_{L}(Y_L)] \leq U_L$, so $XU_L$ is normal in $L$. As $L$ is solvable, $C_L(Y_L) = C_X(Y_L)Q_L$ and either $C_X(Y_L) = 1$ or $X \cong 3^{1+2}_+$. The latter case is impossible as $W$ is quadratic on $U_L/Y_L$. Hence $U_L=[U_L,O^2(L)]$  and $U_L/Y_L$ contains  no central $L$-chief factors. We know that every $L$-chief factor in $U_L/Y_L$ is a   wreath product module for $\SL_2(2) \wr 2$ with $\ov W$ acting quadratically.  In particular, for every non-central chief factor $F$ of $L$ in $U_L/Y_L$ we have $[F,\ov W ] = [F, Z(\ov{Q})]$.  Set $W_1= [W,D]$.  Then $$\ov {W_1} \geq  [\ov W, \ov Q]=  Z(\ov Q).$$  Hence  $[F,W]= [F,W_1]$  for every non-central chief factor $F$ of $L$ in $U_L/Y_L$. Set $\wt L = L/Y_L$ and let $z \in Q$ with $Z(\ov{Q}) = \langle \ov{z} \rangle$. As $C_F(Z(\ov{Q})) = [F,Z(\ov{Q})]$ for each $F$, we have $C_{\wt{U_L}}(z) = [\wt{U_L},z]$; then as $W$ acts quadratically on $\wt{U_L}$, we have
$[W,\wt{U_L}]= C_{\wt{U_L}}(W)$.  Thus  $[U_L,W]Y_L=[U_L,W_1]Y_L$.
In particular, $$[W/W_1,U_L]= [U_L,W]W_1/W_1= (Y_L \cap Q)[U_L,W_1]W_1/W_1=(Y_L\cap Q)W_1/W_1$$ and  so $U_L$ acts quadratically on $W/W_1$.  Therefore $U_LC_{D}(W/W_1)/C_{D}(W/W_1)$ is elementary abelian. Hence $$Y_L= U_L' \le C_{D}(W/W_1).$$  Set $R= \langle Y_L^D\rangle$. Then,  as $Y_L \not \le O_2(D)$ by \fref{lem:w1} (i), $Y_L \cap O_2(D)= Y_L \cap Q \le W$ and so $R$ centralizes $O_2(D)/W$ and $W/W_1$.
\fref{lem: cent W/[W,Q]} yields $Y_L \le O_2(D)$ and this  contradicts \fref{lem:w1} (i). We have shown $Q_L = Y_L$. \qedc

  \begin{claim}\label{clm:NGQ}  $|S| = 2^7$ and $N_G(Q)/Q \cong \Sym(3).$
\end{claim}

\medskip

Since $Q_L =Y_L = V_L$ and $\ov Q \cong \Dih(8)$, $|S|= 2^7$ and $|Q| = 2^6$. Then $N_G(Q) = SX$, where $X$ is a Hall $2^\prime$-subgroup of $N_G(Q)$ and $QX$ is normal in $N_G(Q)$. Furthermore $W$ is extraspecial of order $2^5$. As $W/Z = J(Q/Z)$, we have $W$ is normal in $N_G(Q)$. Hence $X$ acts faithfully on $W$ and embeds in $\OO^+_4(2)$. As $[\ov{W}, \ov{Q}] = Z(\ov{Q})$, $S/W$ is faithful on $W/Z$, so   $N_G(Q)/W$  embeds into $\OO_4^+(2)$. Because $\OO_4^+(2) \cong \Sym(3) \wr  2$, and $O_2(N_G(Q)/W) \not= 1$, we get the claim. \qedc

Taking $T \in \syl_3(L)$, we have $N_L(T)$ is a complement to $Q_L$ and so $L= Q_LN_L(T)$ is a split extension of $Q_L$ by $\mathrm {O}_4^+(2)$.  In particular, the isomorphism type of   $S$ is uniquely determined. As $\Sym(8)$ has a subgroup isomorphic to $L$ and $\Sym(8)$ has odd index in $\Alt(10)$, we have $S$ is isomorphic to  a Sylow $2$-subgroup of $\Alt(10)$.

 Let $z \in C_{Y_L}(Q)^\#$, then as $Y_L$ is a $+$-type space for $L$, there is a fours group $A$ of $Y_L$ which has all non-trivial elements $L$-conjugate to $z$. Since $C_G(z)$ has characteristic $2$,  $C_{O(G)}(z) =1$ and so   by coprime action $$O(G)=\langle C_{O(G)}(a)\mid a \in A^\#\rangle=1.$$

  Assume that $G$ has no subgroup of index two. Then $S$ is isomorphic to a Sylow $2$-subgroup of $\Alt(10)$. Therefore \cite[Theorem 3.15]{Mas} implies that  $F^\ast(G)\cong \Alt(10)$, $\Alt(11)$,  $\PSL_4(r)$, $r \equiv 3\pmod 4$, or $\PSU_4(r)$, $r \equiv 1 \pmod 4$.  Notice that $Z(Q)= C_{Y_L}(Q)=\langle z\rangle$ and so $C_G(z)= N_G(Q)$ has  characteristic $2$.   In $\Alt(11)$, $z$ corresponds to $(12)(34)(56)(78)$ and so $C_G(z) \leq (\Alt(8) \times Z_3):2$, which  implies that $C_G(z)$ is not of characteristic 2. In the linear and unitary groups $C_G(z)$ has a normal subgroup isomorphic to $\SL_2(r) \circ \SL_2(r)$, and this contradicts \fref{clm:NGQ}.  Hence $G \cong \Alt(10)$.

 Assume now that $G$ has a subgroup of index two. As   $V_L \leq G^\prime$ we also have $W \leq G^\prime$. Therefore $(G^\prime \cap L)/Y_L \cong \Omega^+_4(2)$ and so $G'$ has  Sylow 2-subgroups isomorphic to those of $\Alt(8)$.   Applying \cite[Corollary A*]{GoHa} we have $F^\ast(G)  \cong \Alt(8)$, $\Alt(9)$ or $\PSp_4(3)$.   Again in $G' \cong \PSp_4(3)$, we have that $G'$ contains a subgroup of shape $\SL_2(3) \circ \SL_2(3)$. This contradicts \fref{clm:NGQ} and proves the proposition.
\end{proof}

\section{$\ov {L^\circ} \cong \GammaSL_2(4)$}

In this section we attend to the case from \fref{prop:unambiguous}(ii). Hence we have $p=2$, $\ov {L^\circ} \cong \Gamma \SL_2(4)$,  $V_L$ is the natural $\SL_2(4)$-module and either $Y_L=V_L$ or $|Y_L/V_L|=2$.  Notice that as $L \not \le N_G(Q)$ and $L$ centralizes $Y_L/V_L$, if $Y_L> V_L$, $Y_L$ does not split over $V_L$ and $C_{Y_L}(Q) =C_{V_L}(Q)$ has order $2$.  Furthermore, $C_S([Y_L,Q])= Q_L$.

Our aim is to prove

\begin{proposition}\label{prop:S5-1}  Suppose  $L \in \mathcal{L}_G(S)$ and $L \not \le N_G(Q)$ with  $\ov L$ in the unambiguous wreath product case.  If $Y_L \not \le Q$ and $\ov {L^\circ} \cong \GammaSL_2(4)$, then $G \cong \Mat(22)$ or $\Aut(\Mat(22))$.
\end{proposition}

Notice that as $Q_L \in \syl_2(C_L(Y_L))$, $C_L(Y_L)/Q_L$ is centralized by $L^\circ$, and so $C_{L^\circ}(Y_L)=Q_L \cap L^\circ$ as the Schur multiplier of $\SL_2(4)$ has  order $2$. We also have   $|\ov Q|\ge 4$ and $|Z(Q) \cap V_L|=2$.

\begin{lemma}\label{lem:fusioncontrol} For $N = N_G(Q_L)$ we have $(Z(Q) \cap V_L)^N \cap Y_L \subseteq V_L$.  In particular, $N$ normalizes $V_L$.
\end{lemma}

\begin{proof}  If $V_L = Y_L$, there is nothing to prove. Assume that $|Y_L : V_L| = 2$. Choose $g \in N$, put $U=(Z(Q) \cap V_L)^g$ and  assume that $U\not\leq V_L$. Recall that $Y_L = \Omega_1(Z(Q_L))$ and so $U\leq Y_L$  and $Y_L$ is normalized by $N$.   Then  $C_{L}(U)C_N(Y_L)/C_N(Y_L)  \cong 5:4$ or $2 \times \Sym(3)$.   As $C_N(U^{g^{-1}})$ normalizes $Q\cap Y_L$,  $C_N(U^{g^{-1}})$ is not irreducible on $Y_L/U^{g^{-1}}$.  This excludes the possibility $C_L(U)C_N(Y_L)/C_N(Y_L)\cong 5:4$ which is irreducible on $Y_L/U$. Hence we see that $Z(Q) \cap V_L$ has exactly $15 + 10 = 25$ conjugates under $N$, but 25 does not divide the order of $\SL_5(2)=\Aut(Y_L)$.   This contradiction proves the lemma.
\end{proof}

\begin{lemma}\label{lem:S5-2} We have  $Q_L= Y_L$ and either
\begin{enumerate}
\item $|S|= 2^7$,  $L/Q_L\cong\Gamma \SL_2(4)$, $N_G(Q)/Q \cong \SL_2(2)$, there exists a subgroup $E \le S$ of order $2^4$ which is normalized by $N_G(Q)$ such that $N_G(E)/E \cong \Alt(6)$ and $N_L(E)$ has index $5$ in $L$. Furthermore all the involutions in $\langle N_G(E), L\rangle$ are conjugate.
\item $G$ has a subgroup of index $2$ which satisfies the conditions in \rm (i).
\end{enumerate}
\end{lemma}

\begin{proof}   We have $\ov {S} \cong \Dih(8)$ and $\ov Q \not \le \ov S_0$ as $\ov {L^\circ} \cong \GammaSL_2(4)$.  \fref{lem:not base} implies that $\ov W \not \le \ov{S_0}$. By assumption,  we either have $Y_L = V_L$ or $|Y_L:V_L|=2$. In particular, $2^4 \le |Y_L| \le 2^5$.  Since $\ov Q$ is normal in   $\ov S$ and contains $\ov W$ we know

\begin{claim} Either
$\ov{Q}$ is elementary abelian of order $4$ or $\ov Q= \ov S$
\end{claim}

\medskip

 As $V_L$ is a natural $\SL_2(4)$-module and $L \not \le N_G(Q)$,  we have $C_{Y_L}( Q)= C_{Y_L}(S)$ has order $2$ and $[Y_L,Q]=[V_L,Q]$ has order $8$. Furthermore, as $W$ is normal in $S$ and is not contained in $S_0$, we have   $[Y_L,Q,W]= Z$ where $Z= C_{V_L}(S)$ has order $2$.  Thus, arguing exactly as before \fref{clm:qlyl} and in the proof of \fref{clm:W4} we obtain

\begin{claim}\label{clm:S5-W quad} $|\ov W|=4$, $[W,W]= Z$ and $[Q_L,W,W]\le Y_L$.
\end{claim}

\medskip

\begin{claim}\label{clm:S5-YMnotQM}
Assume that $Q_L> Y_L$. Then $[Q_L,O^2(L)]\not\le Y_L$.
\end{claim}

\medskip

Suppose that  $[Q_L,O^2(L)] \le Y_L$. Then $V_L \not\le \Phi(Q_L)$ by Burnside's Lemma \cite[Proposition 11.1]{GoLyS2}, which contradicts \fref{lem:VL sub QL'}(iii). This   proves the claim\qedc

\medskip

\begin{claim}\label{clm:S5-V_LnotY_L} If $V_L < Y_L$, then $\ov Q = \ov S$.
\end{claim}

\medskip

If $\ov Q$ has order $4$, then $\ov Q = \ov W$ by \fref{clm:S5-W quad}, so $\ov Q$ normalizes a Sylow $3$-subgroup $\ov T$ of $\ov L$ and so $Q$ normalizes $C_{Y_L}(T)$ which has order $2$ and complements $V_L$. Hence $C_{Y_L}(T) \leq Z(Q)$, so $T \le N_G(Q)$ and therefore $L=\langle T ,S\rangle \le N_G(Q)$, a contradiction.  Thus $\ov Q= \ov S$ has order $8$.\qedc

\medskip
\begin{claim}\label{clm:Q_L=Y_L} We have $Q_L = Y_L$.
\end{claim}

\medskip
Suppose false.   By \fref{clm:S5-W quad} $W$ acts quadratically  on $Q_L/Y_L$  and $|\ov W|=4$.  Also $\ov W \not\leq \ov{S_0}$, so \fref{lem:Sym5-modules}  implies that the non-central $ L $-chief factors in  $Q_L/Y_L$ are orthogonal modules for $\ov{L} \cong \OO_4^-(2)$. In particular, as $ L$-modules, the non-central $L$-chief factors of $Q_L/Y_L$ are not isomorphic to $V_L$.

  Choose $E \le Q_L$ normal in $L$ and minimal so that $E/Y_L$ contains a non-central $L$-chief factor and let $F$ be the preimage of $C_{E/Y_L}(O^2(L))$.  Then $[F,O^2(L)] \le Y_L$ and \fref{lem: U>YL abelian} applies to yield $F \le Y_L$. In particular, $[E,E]\le Y_L$.

  We claim $E'\le V_L$.  This is obviously the case if $V_L= Y_L$. So suppose that $|Y_L:V_L|=2$. If $E'\not\le V_L$. Then the minimal choice of $E$ and $E'V_L= Y_L$ implies that $E/V_L $ is extraspecial of order $2^5$.  Notice that $[E,W] \le W$ and $W/Z$ is elementary abelian as $[W,W] = Z$ by \fref{clm:S5-W quad}. Hence, as $[E,W]Y_L/V_L$ has order $2^3$, we infer that $E/V_L$ has $+$-type contrary to $\ov {L} \cong \Gamma \SL_2(4)$.   Hence $E/V_L$ is elementary abelian. If $[Q_L,E]= V_L$, then $E/V_L$ has order $2^4$ by \fref{lem:Sym5-modules}  and so $Q_L/C_{Q_L}(E)$ embeds into $$\Hom_L(E/V_L,V_L)\cong (E/V_L)^*\otimes V_L \cong (E/V_L)\otimes V_L$$ by \fref{lem:dualchieffactors}. Since $Q_L/C_{Q_L}(E)$ involves only trivial and orthogonal modules this contradicts \cite[Lemma 2.2]{Prince}.

  Thus $[E,Q_L]= Y_L>V_L$.

   By \fref{clm:S5-V_LnotY_L}
\begin{center} $\ov Q= \ov S$ has order $8$.\end{center}
In summary we now know  $|\ov W| =4$ and $\ov {[W,Q]} =\ov{[W,S]} = Z(\ov S)$.

We calculate using $Z$ is normal in $D$ by \fref{clm:S5-W quad} that $$[W,Q,Q]= \langle[V_L,Q,Q,Q]^{D} \rangle = \langle Z ^{D} \rangle =Z.$$
Therefore $$[E,[W,Q],Q] \le E\cap [[W,Q],Q] \le Z \le Y_L.$$
 As $|[Z(\ov{S}),E/Y_L]| = 4$ and $\ov Q = \ov S$, this implies that $|C_{E/Y_L}(\ov S)| =4$.
As $E/Y_L$ is the orthogonal $\OO_4^-(2)$-module for $L$,   this is impossible.  We have  proved the claim.
\qedc

\begin{claim}\label{clm:structure}
Suppose that $Y_L= V_L$.  Then $L$ is a maximal 2-local subgroup of $G$, $N_G(Q)/Q \cong \SL_2(2)$, there exists a subgroup $E \le S$ of order $2^4$ which is normalized by $N_G(Q)$ such that $N_G(E)/E \cong \Alt(6)$ and $N_L(E)$ has index $5$ in $L$.
\end{claim}

\medskip

By \fref{clm:Q_L=Y_L} we have $|S|= 2^7$, and $|\ov W|= 2^2$. Also $|[W,Y_L]| = 8$ and $Y_L \not\le Q$, so $Q \cap Y_L = [W,Y_L] = W \cap Y_L$, Therefore  $| W|=2^5$. Set $C= C_Q(W)$. Then $C$ centralizes $[Y_L,Q]$ which has order $2^3$ and so $C\le C_{L}([Y_L,Q])= Y_L$. Thus $C\le C_{Y_L}(W)$ which has order $2$.   Then, by \fref{clm:S5-W quad},  $W'= Z=C$ and, as $W$ is generated by involutions, we have $W$ is extraspecial.  Since $[Y_L,Q]\le W$, $W$ has $+$-type.

Observe $W/Z = J(Q/Z)$, so $W$ is normal in $N_G(Q)$ and  $N_G(Q)/Z$ embeds into $\Aut(W)\cong 2^4{:}\OO_4^+(2)$.

 Assume that $Y_L Q/Q$ normalizes a subgroup  $T$ of $O_3(N_G(Q))/Q$ which has fixed points on $W/Z$.  Then $W= [W,T]C_W(T)$ and $[W,T]\cong C_W(T) \cong \Q_8$ and these subgroups are normalized by $Y_L$. But then $$[W,Y_L] = [C_W(T),Y_L][W,T,Y_L].$$ Since $[W,Y_L]$ is elementary abelian and $\Omega_1(P) = Z(P)$ for $P \cong \Q_8$, we conclude that    $$[C_W(T),Y_L]=[W,T,Y_L]= Z$$ and then $[W,Y_L]$ has order $2$ which is nonsense as $Y_L$ is the natural module. Therefore $Y_L$ normalizes no such subgroup.

Let $F= O_{2,3}(N_G(Q))$. Assume that $|F/Q|= 9$.  Then the previous argument implies that $C_{F/Q}(Y_L) \not= 1$. Let $T_1$ be the preimage of this subgroup. Then $[Y_L,Q]$ is normalized by $T_1$. Hence $Y_L= C_{Y_LQ} ([Y_L,Q])$ is normalized by $T_1$. Using the fact that $Q$ is weakly closed in any $2$-group which contains it, for $w \in Y_L^\#$, we let $Q_w$ be the unique conjugate of $Q$ in $O_2(C_G(w))$. Then $T_1$ permutes the elements of $Y_L$ and so $T_1$ normalizes $L^\circ  =\langle   Q_w\mid w \in Y_L^\#\rangle$. Since $L=L^\circ Y_L$, we have that $T_1$ normalizes $L$.
On the other hand,
$WY_L$ is normalized by $T_1$ and, as $T_1$ acts fixed-point freely on $W/Z$, $T_1$ acts transitively on $WY_L/Y_L \cong W/[Y_L,Q]\cong 2^2$ and this is impossible as $W \cap O^2(L)$ is a maximal subgroup of $W$ and is normalized by $T_1$.

Hence $|F/Q|= 3$, $N_G(Q)= FS$ and $N_G(Q)/Q\cong \SL_2(2)$.   In particular, $|Q|= 2^6$, $S= Y_LQ$, and $FY_L/W \cong 2 \times \SL_2(2)$. It follows that
$$[W,Q] \text{ is elementary abelian of order }8.$$

et $E= C_S([W,Q])$. As $W$ is normal in $N_G(Q)$, so is $E$. As $|S|= 2^7$ and $|\GL_3(2)|_2=2^3$, we have $|E| \ge 2^4$. Since $F$ acts fixed-point freely on $W/Z$ (being normalized by $Y_L$), we have $E \le Q$ and then  $E$ is normal in $N_G(Q)$. Since $E \cap W= [W,Q]$, we find $|E| = 2^4$. Let $S\le L_1 \le L$ be such that $L_1/Q_L \cong \Sym(4)$ has index $5$ in $L$. Notice that $O_2(L_1)= S_0$. Then $E \le C_{L}([Y_L,Q,Q]) = Y_LS_0$. Also $Y_L \leq S_0$, so $S_0 = Y_LS_0$. Therefore $E \le S_0$. Now $EY_L/Y_L$  acts as a Sylow $2$-subgroup of $\SL_2(4)$ on the natural module. In particular for any involution  $e \in E \setminus Y_L$ we have that $C_{Y_L}(e) = E \cap Y_L$. This implies that all involutions in $EY_L$ are contained in $Y_L \cup E$ and therefore $E$ and $Y_L$ are the only elementary abelian subgroups of $S_0$  of order $2^4$. In particular,  $L_1$ normalizes $E$. Now $N_G(E) \ge \langle L_1,N_G(Q)\rangle \in \mathcal L_G(S)$. Notice that $L_1$ has orbits of lengths $3$, and $12$ on $E$ and that $N_G(Q)$ does not preserve these orbits. Hence $N_G(E)$ acts transitively on $E^\#$. As $N_G(Q)= C_G(Z)$, we now have that $|N_G(E)|=15|N_G(Q)|= 2^7\cdot 3^2 \cdot 5$.  We have that $X = N_G(E)/E$ is isomorphic to a subgroup of $\GL_4(2) \cong \Alt(8)$ of order $2^3\cdot 3^2 \cdot 5$. We consider the action of $X$  on a set of size $8$. As $\Alt(8)$ has no subgroups of order 45,  $X$ is not transitive. Hence $X$ is isomorphic to a subgroup of $\Alt(7)$, $\Sym(6)$ or $X \cong (\Alt(5) \times  3){:}2$. Suppose that $ X \cong (\Alt(5) \times  3){:}2 $. As $N_G(Q)/Q \cong \Sym(4)$, we see that $EQ/E \leq \Alt(5)$. Since $E$ is the natural  $\SL_2(4)$-module, we get that $|Z(Q)| = 4$. But, by \fref{clm:S5-W quad}, $|Z(Q)| = 2$. Hence we have one of the first two possibilities and then obviously $X = N_G(E)/E \cong \Alt(6)$.

We just have to show that $L$ is a maximal 2-local subgroup. Let $M$ be a 2-local subgroup with $L \leq M$. As $Q \leq M$, we have that $M$ is of characteristic 2. Then $Y_L = Y_M$ and  $C_G(Y_L) = Y_L$.  As $|N_G(Q) : S|=3$ and $Y_L$ is not normal in $N_G(Q)$, we have $N_{M}(Q) = S = N_L(Q)$.  As $L$ acts transitively on $Y_L^\#$, we conclude $M = N_M(Q)L = N_L(Q)L = L$.
\qedc

 \begin{claim} If $Y_L = V_L$, then $G$  has just one conjugacy class of involutions.\end{claim}

\medskip

By \fref{clm:structure} $N_G(E)/E \cong \Alt(6)$.  As $Y_L \not\leq E$, there is an involution $y \in Y_L \setminus E$. Now $y$ inverts an element of order 5 in $N_G(E)$ and so $|[E,y]| = |C_E(y)| = 4$. This shows that all involutions in $Ey$ are conjugate. As all involutions in $S/E$ are conjugate in $\Alt(6)$ and all the involutions in $Y_L$ are $L$-conjugate, this proves the claim.\qedc

We have now  proved that  (i) holds when $Y_L = V_L$.

\begin{claim}\label{clm:56}
Suppose that $Y_L> V_L$.  Then  $G$ has a subgroup of index $2$.
\end{claim}

\medskip

We have that $|S|= 2^8$.  By \fref{clm:S5-V_LnotY_L} and \fref{clm:Q_L=Y_L},  $S=QY_L$.  We are going to show that $J(S) = Y_L$.
For this let $A \leq S$ be elementary abelian of maximal order and assume that $A \ne Y_L$. Then $|AY_L/Y_L|  \leq 4$. As there are no transvections on $V_L$, we get $|AY_L/Y_L| = 4$ and we may assume that $A$ acts quadratically on $Y_L$ by \cite[Theorem 25.2]{GoLyS2}.   As $W \not\leq S_0$ by \fref{lem:not base} and $|\ov W| = 4$ by \fref{clm:S5-W quad},  $W$ does not act quadratically on $Y_L$, $AY_L/Y_L \le S_0/Y_L$ and $S_0 = AY_L$. Now $A \cap Y_L$ has order 8 and so $|C_{Y_L}(S_0)| = 8$. But $(L^\circ)^\prime$ is generated by two conjugates of $S_0$, which gives $C_{Y_L}(L^\circ) \not= 1$ a contradiction to  \fref{lem:VL sub QL'}(i).
Thus $Y_L= J(S)$ is the Thompson subgroup of $S$. In particular, $N_G(Y_L)$ controls $G$-fusion of elements in $Y_L$.
  As $S \in \syl_2(G)$ and $C_S(Y_L) = Q_L$,  $Q_L \in \syl_2(C_G(Y_L))$ and we have $N_G(Y_L) = C_G(Y_L)N_{N_G(Y_L)}(Q_L)$. By \fref{lem:fusioncontrol} $$V_L\text{ is normal in }N_G(Y_L).$$
Suppose that $O^2(L) \ge Y_L$.  Then $O^2(L)/V_L \cong \SL_2(5)$ has quaternion Sylow $2$-subgroups and $|L:O^2(L)|=2$. On the other hand,
there exists  $g \in N_G(Q) \setminus N_G(Y_L)$ with $WY_L \ge (Y_L^g\cap Q)Y_L\ne Y_L$ and $(Y_L^g\cap Q)V_L/V_L$ is elementary abelian, which is a contradiction.  Therefore $O^2(L)/V_L\cong \SL_2(4)$ and,  as $W$ does not act quadratically on $Y_L$, we see that $|W : W \cap O^2(L)| = 2$ and thus  $O^2(L)W/V_L \cong \Gamma\SL_2(4)$.  Hence $L$ has a subgroup $L_0=O^2(L)W$ of index $2$ with $Y_L \cap L_0 = V_L$.

 Let $T \in \syl_2(L_0)$ and  $w \in Y_L \setminus T$.  Suppose that for some $x \in G$,  $w^x \in T$  and $|C_{S}(w^x)| \ge |C_S(w)|$. As $L^\circ$ has orbits of length 6 and 10 on $Y_L \setminus V_L$, we may assume $|C_{S}(w^x)| \geq |S|/2$. But then as $V_L$ is the natural module, it does not admit transvections and so $w^x \in V_L$. As $N_G(Y_L) = N_G(V_L)$ and $N_G(Y_L)$ controls fusion in $Y_L$, this is not possible. Hence the supposed condition cannot hold. Application of \cite[Proposition 15.15]{GoLyS2}, shows that $G$ has a subgroup  of index $2$. This proves \fref{clm:56}.\qedc
\medskip

Let $G_0$ be a subgroup of $G$ of index $2$, and set $Q_0= Q \cap G_0$. We have $V_L \le L^\circ \le G_0$. Hence $W=\langle [V_L,Q]^{D }\rangle \le G_0$.  In particular, $W \le Q_0$ and so  $Z(Q_0) = Z$ and   $Q_0$ is large in $G_0$. Set $L_0= O^2(L)Q_0 = O^2(L)W$. Then $L_0^\circ /V_L \cong \Gamma \SL_2(4)$   and $Y_{L_0}= V_{L_0} = V_L \not \le Q_0$. Thus $(G_0,L_0)$ satisfies the hypotheses of (i). This proves (ii) holds if $V_L \not= Y_L$.
\end{proof}

\begin{proof}[Proof of \fref{prop:S5-1}:] By \fref{lem:S5-2} we just have to examine  the structure in \fref{lem:S5-2}(i), so we may assume that \fref{lem:S5-2}(i) holds.
\\
\\
 By \fref{lem:splitA6}
$$N_G(E)\text{ splits over }E.$$

As $N_G(Q) \leq N_G(E)$,   for a 2-central involution $z$ we have that $C_G(z)$ is a split extension of $E$ by $\Sym(4)$. As $O(C_G(z)) = 1$   coprime action yields $O(G) = \langle C_{O(G)}(e)\mid e \in E^\# \rangle = 1$.  In particular $F(G) = 1$ and $E(G)\ne 1$. Suppose that $J^*$ is a non-trivial subnormal subgroup of $G$ normalized by $\langle L, N_G(E)\rangle$.  Then $S \cap J^* \ne 1$.
Since  $1 \ne J^* \cap N_G(E)$ is   normal   in $N_G(E)$ and  $1\ne J^* \cap L$ is normal in $L$, it follows that $J^* \cap N_G(E)\ge J^* \cap S \ge EY_L$. Hence  $J^* \ge \langle Y_L^{N_G(E)} \rangle = N_G(E) \ge S$ and $J^* \ge \langle S^L\rangle =L$. Therefore there is a unique non-trivial subnormal  subgroup of $G$ of minimal order normalized by $\langle L, N_G(E)\rangle$. It follows that $\langle L, N_G(E)\rangle$ is contained in a component $J$ of $G$. Since $O(G)=1$ and $S \le J$, $J= E(G)$.
  As $J$ has just one conjugacy class of involutions by \fref{lem:S5-2}(i) and, for $z \in E^\# $, $ C_G(z) \le N_G(E)$, it follows that $G=J$ is simple. Using  $G$ has just one conjugacy class of involutions and applying \cite[Theorem]{Janko} yields $G \cong \Mat(22)$. This proves the proposition when \fref{lem:S5-2}(i) holds. If \fref{lem:S5-2}(ii) holds, then $G \cong \Aut(\Mat(22))$.
\end{proof}

\section{$\ov {L^\circ} \cong \SL_2(4)$}\label{sec:L24}

In this section we investigate  the configuration in \fref{prop:unambiguous}(iii). Thus  $\ov {L^\circ} \cong \SL_2(4)$, $|Y_L : V_L| = 2$ and $V_L$ is the natural $\SL_2(4)$-module.

As $Q \leq L^\circ$, $C_{V_L}(S_0) = C_{V_L}(Q) \leq Z(Q)$, so $Q$ is normal in $N_{L^\circ}(C_{V_L}(S_0))$ and hence  $\ov Q= \ov{S_0}$ is a Sylow 2-subgroup of $\ov{L^\circ}$. In particular $Z(Q) \cap Y_L = Z(Q) \cap V_L$ is of order 4.

\begin{lemma}\label{lem:Qabel} The subgroup $Q$ is elementary abelian. In particular, $Q \cap Y_L = Q \cap V_L = C_{Y_L}(Q)=Z$,  $|Y_LQ/Q|=2^3$ and $|V_LQ/Q|=2^2$.
\end{lemma}

\begin{proof} We know that $[Q,V_L] = C_{V_L}(Q)= Q\cap V_L$ and, as $\ov Q$ is elementary abelian, $\Phi(Q) \le Q_L$.  If $\Phi(Q) \not= 1$, then, since $Z(S) \cap \Phi(Q) \not= 1$, we deduce $\Phi(Q) \cap V_L \not= 1$. As $N_L(QQ_L)$ normalizes $Q$ and is irreducible on $[V_L,Q]$, $[V_L,Q] \leq \Phi(Q)$.  But then $V_L$ centralizes $Q/\Phi (Q)$, so $V_L \leq O_p(N_G(Q))= Q$, a contradiction. This shows $Q$ is elementary abelian and then also $Y_L \cap Q = V_L \cap Q= C_{Y_L}(Q)$.
\end{proof}

\begin{proposition}\label{prop:L24} Suppose  $L \in \mathcal{L}_G(S)$ and $L \not \le N_G(Q)$ with  $\ov L$ in the unambiguous wreath product case.  If $Y_L \not \le Q$,  $\ov {L^\circ} \cong \SL_2(4)$ and $|Y_L:V_L|=2$, then $G$ is  $\Aut(\Mat(22))$.
\end{proposition}

\begin{proof} We start by observing that the action of $L$ on $Y_L$ gives

\begin{claim} \label{clm:1}\begin{enumerate} \item $|V_LQ/Q|=|Q:C_Q(V_L)|=2^2$;
 \item for all $v \in V_L \setminus Q$, $C_Q(v) = C_Q(V_L)$; and
 \item  for all $w \in Q \setminus Q_L$,  $[w,V_L] = [Q,V_L]$. \end{enumerate}
 \end{claim}
\medskip

Let $B= N_L(QQ_L)$. Then $B $ contains an element $\beta$ of order $3$ which acts fixed-point freely on $V_L$ and irreducibly on $[V_L,Q]=C_{Y_L}(Q)$.

 Using \fref{clm:1} (ii) and \fref{lem:soldual} yields $[V_L,F(N_G(Q)/Q)] = 1$.
  Let $K\ge Q$ be the preimage of $$ [E(N_G(Q)/Q),V_LQ/Q].$$  Then $K$ is non-trivial,  normalized by $B$ and \fref{lem:soldual} implies $V_LQ/Q$ acts faithfully on $K/Q$.

 The three involutions of $QQ_L/Q_L$ each centralize a subgroup of $Y_L$ of order $2^3$ and by \fref{lem:VL sub QL'}(i), there are three elements of $Y_LQ/Q$ which act on $Q$ as $\GF(2)$-transvections, they generate $Y_LQ/Q$ and are permuted transitively by $B/Q$. As $B$ normalizes $K$ and as $V_LQ/Q$ acts faithfully on $K/Q$, at least one and hence all of the transvections in $Y_L Q/Q$ act faithfully  on $K/Q$.

 If  $C_Q(K) \ne 1$, then $C_{C_{Q}(K)}(S) \ne 1$.  As $\Omega_1(Z(S))= C_{V_L}(S)$ by \fref{lem:VL sub QL'} (ii), and $C_Q(K)$ is normalized by $B$, we have $[Q, V_L] \le C_Q(K)$. But then $K=\langle V_L^K\rangle Q$ centralizes $Q/C_Q(K)$ contrary to $C_K(Q)=Q$. Hence $C_Q(K)=1$.

 Let $V$ be a non-trivial minimal $KY_L$-invariant subgroup of $Q$.
 Then $KY_L$ acts irreducibly on $V$. Moreover,  as $Y_L$ does not centralize $V$, $V \not \le Q_L$ and, as $V_L$ is the natural $\ov{L^\circ}$-module we have  $[Y_L,V]= [Y_L,Q] = Y_L \cap Q \le V$.  It follows that $K$ centralizes $Q/V$ and so $K/Q$ acts faithfully on $V=[Q,K]$ which is normalized by $B$. Hence $C_{Y_L}(V)= Y_L \cap V= Y_L \cap Q$ and  $Y_LQ/Q$ acts faithfully on $V$. Recall that  $ Y_LQ/Q$ is generated by elements which operate as transvections on $Q$ and hence on $V$. Therefore  \cite[Theorem]{McL} applies to give $KY_L/Q \cong \SL_m(2)$ with $m \ge 3 $, $\Sp_{2m}(2)$ with $m \ge 2$, $\OO_{2m}^\pm(2)$   with $m \ge 2$, or $\Sym(m)$ with $m \ge 7$. Furthermore, $V=[Q,K]$ is the natural module in each case.

 Since $C_{Y_LQ/Q}(S/Q)$ contains a transvection and has order $2^2$, $KY_L/Q \not \cong \SL_m(2)$ with $m \ge 3$ or $\OO_{2m}^\pm(2)$ with $m \ge 2$.  Suppose that $KY_L/Q \cong \Sym(m)$ with $m \ge 7$.  Then, as $Y_LQ/Q$ is generated by three transvections, we see that $Y_LQ/Q$ is generated by three commuting transpositions  in $KY_L/Q$.  Let $t$ be the product of these transpositions.  Then, as $m \ge 7$, $|[V,t]| = 2^3$.  However, $|[V,Y_L]|=2^2$, and so we have a contradiction.  We have demonstrated

\begin{claim}\label{clm:2} $KY_L/Q \cong \Sp_{2m}(2)$, $m \geq 2$ and  $[Q,K]=[Q,K Y_L]$ is the natural module.\end{claim}
\medskip

Since $[Q,K]$ is the natural $KY_L /Q$-module and $[V_L,Q] \le [Q,K]$ has order $2^2$, we have $[[V_L,Q],S]\ne 1$.  In particular, $QQ_L/Q_L < S/Q_L \cong \Dih(8)$ and $SQ/Q\cap K/Q$ acts non-trivially on $[Q,V_L]$.

 Consider $Q^*=O_2(KS)$.  Since $Q^*$ centralizes $[Q,K]$, $Q^*$ centralizes $[V_L,Q]$ and so $Q^*Q_L= QQ_L$.  Hence $\Phi(Q^*)\le Q_L$. If $\Phi(Q^*) \ne 1$, then $$[Q,K]=\langle \Omega_1(Z(S))^{K}\rangle \le \Phi(Q^*)$$ and so also $[Q^*,K]=[Q^*,K,K] \le [Q,K] \le \Phi(Q^*)$ which is impossible.  Hence $Q^* $ is elementary abelian and it follows that $Q\le Q^*= C_{Q^*}(Q) \le Q$. Since $KS$ acts on $[Q,K]$ and $KY_L/Q \cong \Sp_{2m}(2)$, we now deduce $S \le KY_L$ from the structure of $\Out(K/Q)$.  Hence $ B= \langle S^B\rangle \le KY_L$ as $B$ normalizes $KV_L$. It follows that $B/Q$ is  the minimal parabolic subgroup $P$ of $K/Q$ irreducible on $[Y_L,V]$ and with $O^2(P)$ centralizing $[Y_L,V]^\perp/[Y_L,V] = C_{Y_L}(V)/[Y_L,V]$. Therefore there is $\beta \in K$ of order three such that $\langle \beta\rangle$ is transitive on the transvections in $Y_LQ/Q$ and normalizes $Q_LQ/Q$ which has index $2$ in $S/Q$. In particular, from the structure of the natural $\Sp_{2m}(2)$-module $\beta$  centralizes $$C_V(Y_L)/[V,Y_L]=(V \cap Q_L)/(V\cap Y_L)= (V\cap Q_L)Y_L/Y_L \leq  [Q_L,V]Y_L/Y_L.$$

   As $V$ is abelian, $V$  acts quadratically on $Q_L/V_L$. By \fref{lem:Sym5-modules}, $Q_L/V_L$ involves only  natural $\SL_2(4)$-modules and trivial modules as $L$-chief factors. We know  $\beta$ acts fixed-point freely on the natural module and so, as $\beta$  centralizes $[Q_L,V]Y_L/Y_L$,   all the $L$-chief factors of $Q_L/V_L$ are centralized by $L$.  In particular,  $V_L$ is the unique non-central $L$-chief factor in $Q$ and so $Y_L \cap \Phi(Q_L) = 1$. As $\Omega_1(Z(S)) \le V_L$ by \fref{lem:VL sub QL'} (ii),  $\Phi(Q_L) = 1$, so $Q_L = \Omega_1(Z(Q_L)) = Y_L$, which together with  $S/Q_L \cong \Dih(8)$ implies

\begin{claim}\label{clm:3} $Y_L = Q_L$ has order $2^5$ and $|S|=2^8$.\end{claim}
\medskip

Together \ref{clm:2} and \ref{clm:3} give

\begin{claim}\label{clm:4}$|Q| = 2^4$ and $N_G(Q)/Q \cong \Sym(6)$.\end{claim}
\medskip

We next show that $G$ has a subgroup of index two.  In $N_G(Q)$ we have a subgroup $U$ of index 2 of shape $2^4.\Alt(6)$. Furthermore $Y_L \not\leq U$ and $V_L \le U$. Since $[v,Q]= C_Q(v)$ for $v \in V_L\setminus Q$ and $U/Q$ has one conjugacy class of involutions, all the involutions in $U\setminus Q$ are $U$-conjugate. Since $L$ acts transitively on $V_L$  and $U$ is transitive on  $Q^\#$, we have that all the involutions in $U$ are $G$-conjugate. As $Q$ is large, we have $C_G(z) \le N_G(Q)$ for $z \in Q^\#$.  Hence all the  involutions in $U$ have centralizer which is a $\{2,3\}$-group. There is an involution $t$  in $Y_L \setminus V_L$, which is not in $U$ and centralized by an element of order 5 in $L$. Hence $t$ is not conjugate to any involution of $U$. Application of \cite[Proposition 15.15]{GoLyS2} gives a subgroup  $G_1$ of index two in $G$. We have $N_{G_1}(Q)/Q \cong \Alt(6)$. By \fref{lem:splitA6} this extension splits and we have that the centralizer of a $2$-central  involution $z \in G_1$ is a split extension of an elementary abelian group of order 16 by $\Sym(4)$. In particular $O(C_G(z)) = 1$ and so by coprime action $O(G) = \langle C_{O(G)}(e)\mid  e \in Q^\# \rangle = 1$. As $Y_L \not\le Q$, there is an involution $y \in  N_{G_1}(Q) \setminus Q$. Since all involutions in $Qy$ and in $N_{G_1}(Q)/Q$ are conjugate,   $G_1$ has just one conjugacy
class of involutions. In particular $F^\ast(G_1)$ is simple. Application of \cite[Theorem]{Janko} gives that $F^\ast(G_1) \cong \Mat(22)$ and so $G \cong \Aut(\Mat(22))$.
\end{proof}

\section*{Acknowledgment}
We thank the referee for numerous comments which have improved the readability and clarity of our work.
The second author was partially supported by the DFG.


\begin{thebibliography}{ABCD}
\bibitem[GH]{GoHa} D. Gorenstein, K. Harada, On finite groups with Sylow $2$-subgroups of type
$A_n$, $n = 8$, $9$, $10$, $11$, Math. Z., 117 (1970), 207--238.
\bibitem[GLS2]{GoLyS2} D. Gorenstein, R. Lyons, R. Solomon, The classification
of the finite simple groups, Amer. Math. Soc. Surveys and Monographs 40(2), (1996).
\bibitem[J]{Janko} Z. Janko, A characterization of the Mathieu simple groups, I, J. Algebra 9, 1968, 1--19.
\bibitem[Mas]{Mas} D. Mason, Finite simple groups with Sylow $2$-subgroup dihedral wreath $Z_2$,
J. Algebra 26, (1973), 10--68.

\bibitem[MSS1]{ov} U. Meierfrankenfeld, B. Stellmacher, G. Stroth, Finite groups of local characteristic p:
an overview in {\em  Groups, combinatorics and geometry}, Durham 2001 (eds. A. Ivanov, M. Liebeck, J. Saxl),
Cambridge Univ. Press, 155--191.
\bibitem[MSS2]{stru} U. Meierfrankenfeld, B. Stellmacher
and G. Stroth,  {\em The local structure theorem for finite groups with a large $p$-subgroup}, Mem.  Amer. Math. Soc. 242, Nr. 1147 (2016).
\bibitem[McL]{McL}J. McLaughlin, Some subgroups of $\SL_n(F_2)$, Illinois J. Math 13, 1969, 105-115.
    \bibitem[PPS]{PPS} Chr. Parker,  G. Parmeggiani, B. Stellmacher,  The $P$!-Theorem, Journal of Algebra 263 (2003), 17--58.
\bibitem[Pr]{Prince} A. R. Prince, On 2-groups admitting $A_5$ or $A_6$ with an element of
 order 5 acting fixed   point freely.
 J. Algebra  49  (1977),  no. 2, 374--386.
\end{thebibliography}
\end{document}